\RequirePackage{fix-cm}
\documentclass[smallextended]{svjour3}       % onecolumn (second format)
\smartqed  % flush right qed marks, e.g. at end of proof
\usepackage{array,tikz,amsfonts,mathtools}
\usepackage[colorlinks=true,linkcolor=blue,citecolor=blue,urlcolor=blue]{hyperref}
\usepackage{subcaption,cite,algorithm,algpseudocode}
\newcommand{\N}{\mathbb{N}}
\newcommand{\R}{\mathbb{R}}

\DeclareMathOperator*{\argmin}{argmin}
\DeclareMathOperator{\relint}{relint}
\DeclareMathOperator{\TV}{TV}
\newcommand{\prox}{\operatorname{prox}}

\newcommand{\dom}{\operatorname{dom}}

\newcommand{\mL}{\mathcal{L}}

\newtheorem{assumption}{Assumption}

\usepackage{bm}

\def \tPlotWidth {.2\paperwidth}
\def \qPlotWidth {.28\paperwidth}
\usepackage{cancel}

\graphicspath{{plots/}{images/}}

\begin{document}

\title{A Preconditioned Version of a Nested Primal-Dual Algorithm for Image Deblurring}

\titlerunning{Preconditioned Nested Primal-Dual Algorithm}        % if too long for running head

\author{Stefano Aleotti \and Marco Donatelli \and Rolf Krause \and Giuseppe Scarlato}

\authorrunning{Aleotti et al.} % if too long for running head

\institute{Stefano Aleotti \and Marco Donatelli \and Giuseppe Scarlato \at
    Department of Science and High Technology, University of Insubria, Como, 22100, Italy\\
    \email{saleotti@uninsubria.it, marco.donatelli@uninsubria.it, gscarlato@uninsubria.it}
    \and
    Rolf Krause \at
    AMCS, CEMSE, King Abdullah University of Science and
    Technology, Thuwal, 23955-6900, Kingdom of Saudi Arabia\\
    Euler Institute, Faculty of Informatics, Università della Svizzera Italiana, Lugano, 6900, Switzerland\\
    Faculty of Mathematics and Informatics, UniDistance, Brig, 3900, Switzerland\\
    \email{rolf.krause@kaust.edu.sa}
}

\date{}
% \date{Received: date / Accepted: date}
% The correct dates will be entered by the editor

\maketitle

\begin{abstract}
    Variational models for image deblurring problems typically consist of a smooth term and a potentially non-smooth convex term. A common approach to solving these problems is using proximal gradient methods. To accelerate the convergence of these first-order iterative algorithms, strategies such as variable metric methods have been introduced in the literature.

    In this paper, we prove that, for image deblurring problems, the variable metric strategy proposed in \cite{NPDIT} can be reinterpreted as a right preconditioning method. Consequently, we explore an inexact left-preconditioned version of the same proximal gradient method. We prove the convergence of the new iteration to the minimum of a variational model where the norm of the data fidelity term depends on the preconditioner. The numerical results show that left and right preconditioning are comparable in terms of the number of iterations required to reach a prescribed tolerance, but left preconditioning needs much less CPU time, as it involves fewer evaluations of the preconditioner matrix compared to right preconditioning. The quality of the computed solutions with left and right preconditioning are comparable.
    Finally, we propose some non-stationary sequences of preconditioners that allow for fast and stable convergence to the solution of the variational problem with the classical $\ell^2$--norm on the fidelity term.
\end{abstract}

\section{Introduction}

We consider the general minimization problem
\begin{equation}\label{eq:problem_formulation}
    \underset{u\in\mathbb{R}^d}{\operatorname{argmin}} \ f(u) + h(Wu),
\end{equation}
where \( f:\mathbb{R}^d\rightarrow \mathbb{R} \) is a convex and smooth function, \( h:\mathbb{R}^{d'}\rightarrow \mathbb{R}\cup\{\infty\} \) is convex and possibly non-smooth, and \( W\in\mathbb{R}^{d'\times d} \) is a linear operator. Problems like \eqref{eq:problem_formulation} frequently occur in inverse problems related to imaging, such as image deblurring, denoising, computed tomography, and others \cite{Bach-2012,Bertero1998b,Chambolle-Pock-2016}. In particular, we focus on image deblurring, described by the model equation
\begin{equation}\label{eq:Model_Equation}
    Au = b^\delta,
\end{equation}
where \( A\in\mathbb{R}^{s\times d} \) represents the discretization of a space invariant convolution operator, \( b^{\delta}\in\mathbb{R}^{s} \) contains the observed image corrupted by white Gaussian noise $\eta_{\delta}$, and \( u\in\mathbb{R}^d \) denotes an unknown two-dimensional image with \( d \) pixels. We assume that \( b^{\delta} \) satisfies
\begin{equation}\label{eq:bdelta}
    b^{\delta} = b + \eta_{\delta}, \qquad \|b^{\delta} - b\| \leq \delta,
\end{equation}
where the vector \( b \) represents the unobserved noise-free data, $\|\cdot\|$ denotes the Euclidean norm, and \( \delta > 0 \) serves as an upper bound of the noise level.

The ill-posed nature of the operator \( A \) and the presence of noise requires a regularization strategy to solve the problem \eqref{eq:Model_Equation}, see \cite{HNO,Engl1996-yp,Bertero1998b}. A common variational approach applied to equation \eqref{eq:Model_Equation} involves solving the optimization problem
\begin{equation}\label{eq:problem_leastsquares}
    \underset{u\in\mathbb{R}^d}{\operatorname{argmin}} \ \frac{1}{2}\|Au - b^{\delta}\|^2 + h(Wu),
\end{equation}
which is a specific instance of the model problem \eqref{eq:problem_formulation} with the differentiable part defined as \( f(u) = \frac{1}{2}\|Au - b^{\delta}\|^2 \). Here, \( f \) is referred to as the \textit{data fidelity} term, measuring the discrepancy between the observed data and the model. The nondifferentiable term \( h \circ W = h(W \cdot)\) acts as a penalty, imposing structure or prior knowledge on the solution. Typically, the influence of these two terms in problem \eqref{eq:problem_leastsquares} is balanced by the action of a regularization parameter \( \lambda > 0 \). For simplicity, in the following, we incorporate the regularization parameter into the definition of the convex function \( h \).

A particularly useful way to compute approximate solutions of the initial problem \eqref{eq:problem_formulation} is to employ proximal gradient methods \cite{Beck-Teboulle-2009b,Combettes-Wajs-2005,Daubechies-et-al-2004}, which are first-order algorithms. These methods achieve mild-to-moderate accuracy while maintaining a low computational cost per iteration, especially when the dimensionality \( d \) is large. The iterative scheme involves alternating a gradient step on the differentiable part \( f \) and a proximal evaluation on the non-smooth term \( h\circ W \), specifically
\begin{align}
    u_{n+1} & = \operatorname{prox}_{\alpha h \circ W}(u_n - \alpha \nabla f(u_n)),
    \label{eq:pg_method}
\end{align}
where $\alpha > 0$ is the step length parameter along the descent direction \( -\nabla f(u_n) \), and \( \operatorname{prox}_{\alpha h \circ W} \) is the proximal operator associated with the non-smooth term \( \alpha h\circ W \). Convergence of the method \eqref{eq:pg_method} to a solution of \eqref{eq:problem_formulation} is guaranteed if \( f \) has an \( L- \)Lipschitz continuous gradient and \( \alpha \) is chosen to be smaller than \( 2/L \) (e.g. \cite{Combettes-Wajs-2005}).

However, proximal gradient methods have two main drawbacks. First, if the chosen step length \( \alpha \) is too small, then the convergence may slow down, particularly when only a rough estimate of the Lipschitz constant \( L \) is available. A potential remedy for this issue is to accelerate the scheme by using either a variable metric approach that incorporates some second-order information of the differentiable part \cite{Bonettini-et-al-2016a,Chouzenoux-etal-2014,Frankel-etal-2015,Ghanbari-2018,Lee2018}, or by adding an extrapolation step that leverages information from previous iterations \cite{Beck-Teboulle-2009a,Ochs-etal-2014}. For example, when \( h \) is the \( \ell^1 \)-norm, a popular algorithm incorporating an extrapolation step is the Fast Iterative Soft Thresholding Algorithm (FISTA) \cite{Beck-Teboulle-2009b}.

The second drawback is that the scheme \eqref{eq:pg_method} assumes that \( \operatorname{prox}_{\alpha h \circ W} \) can be computed in closed form, which is not admissible for several regularization terms, such as Total Variation \cite{Bach-2012,Polson-et-al-2015}. To address this, splitting approaches like primal-dual methods avoid the explicit computation of \( \operatorname{prox}_{\alpha h \circ W} \) by reformulating \eqref{eq:problem_formulation} as a convex-concave saddle point problem \cite{Chambolle10,Malitsky-Pock-2018,Chambolle-et-al-2024}. Moreover, if the convex conjugate \( h^*(v) = \sup_{w\in\mathbb{R}^{d'}}\langle v,w\rangle - h(w) \) has an easily computable proximal operator, then the desired proximal approximation can be computed using a \textit{primal-dual} inner routine \cite{Bonettini-et-al-2016a,Schmidt2011,Villa-etal-2013} that involves only the computation of \( \nabla f(u_n) \), \( \operatorname{prox}_{h^*} \), and matrix-vector products with $W$ and $W^T$ \cite{Bonettini2018a,Chen-et-al-2018,Villa-etal-2013}.

This work focuses on a preconditioning strategy for the variational model~\eqref{eq:problem_leastsquares}. Since the gradient descent method is generally slow, we aim to replace the gradient step in the proximal evaluation with a higher-order iterative method. To this end, one could premultiply the descent direction $-\nabla f(u)$ with a suitable preconditioner $P$. Given that the model \eqref{eq:problem_leastsquares} is closely related to the linear system \eqref{eq:Model_Equation}, standard preconditioning strategies can be applied. Specifically, given an easily invertible matrix $P$ that approximates $A$ and is well-conditioned, one can apply left or right preconditioning to the linear system \eqref{eq:Model_Equation} as follows:
\begin{align*}
    P^{-1}Au  & = P^{-1}b^{\delta}, & \text{(left preconditioning)},  \\
    AP^{-1}Pu & = b^{\delta},       & \text{(right preconditioning)}.
\end{align*}

In the following, we consider proximal gradient methods applied to problem~\eqref{eq:problem_leastsquares}, where acceleration techniques based on left preconditioning and extrapolation are combined, and inexact proximal evaluations are allowed, resulting in an iterative scheme of the form
\begin{equation}\label{eq:pg_VM}
    \begin{cases}
        \bar{u}_n = u_n + \gamma_n(u_n - u_{n-1}) \\
        u_{n+1} \approx \operatorname{prox}_{\alpha_n h \circ W}\left(\bar{u}_n - \alpha_n P_n^{-1}A^T(A\bar{u}_n - b^{\delta})\right),
    \end{cases} \quad \forall \ n\geq 0,
\end{equation}
where \( P_n\in\mathbb{R}^{d\times d} \) is preconditioning matrix and \( \gamma_n \geq 0 \) is the extrapolation parameter. The symbol ``\(\approx\)'' denotes an approximation of the proximal gradient point, which is required whenever the proximal operator cannot be computed in closed form.

While the step length \( \alpha_n \) represents the inverse of a local Lipschitz constant of the gradient, the matrix \( P_n \) aims to capture some second-order information of the smooth part of the objective function. Practical choices for \( P_n \) include regularized versions of the Hessian matrix \cite{Lee2014,Yue2019} or Hessian approximations based on Quasi-Newton strategies \cite{Ghanbari-2018,Jiang-2012,Kanzow-et-al-2022,Lee2018}. The extrapolation parameter \( \gamma_n \) is typically computed according to a predetermined sequence, originally proposed by Nesterov \cite{Nesterov-1983} for smooth problems and later adapted to non-smooth problems by Beck and Teboulle \cite{Beck-Teboulle-2009b}. Lastly, the approximation of the proximal operator in \eqref{eq:pg_VM} is typically achieved by means of a \textit{nested} iterative solver, which is applied, at each iteration, to the minimization problem associated with the computation of the proximal--gradient point.

Many different strategies exist for solving problem \eqref{eq:problem_leastsquares} and, more generally, the variational problem \eqref{eq:problem_formulation}. In \cite{Bonettini-et-al-2023a}, a \textit{Nested Primal-Dual} (NPD) method was introduced for convex composite optimization problems. This algorithm can be seen as a specific instance of the iterative scheme \eqref{eq:pg_VM} with the identity matrix as preconditioner \( P \). It is characterized by using a dual sequence at each iteration to obtain a good approximation of the proximity operator for the non-smooth term. However, the NPD method exhibits slow convergence towards the minimum point, a drawback that can be addressed by employing strategies such as variable metric approaches \cite{NPDIT}.

Variable metric strategies are commonly employed not only to accelerate the convergence of the NPD method but also in primal-dual methods, such as the preconditioned version of the Chambolle-Pock algorithm \cite{Pock-Chambolle-2011}. This algorithm reformulates the initial problem \eqref{eq:problem_formulation} as a saddle point problem and applies a primal-dual strategy, whose convergence is enhanced by using a preconditioner or variable metric. However, \cite{NPDIT} demonstrated that these methods can be surpassed by introducing scaling matrices in a nested framework. Their analysis and comparison focused on solving a Total Variation regularization problem for image deblurring, which takes the form
\begin{equation}\label{eq:TV}
    \argmin_{u\in\mathbb{R}^d} \frac{1}{2}\|Au - b^{\delta}\|^2 + \lambda \TV(u),
\end{equation}
where \( \TV(u) \) denotes the Total Variation functional \cite{Rudin-Osher-Fatemi-1992}. The scaling matrix used in the numerical results in \cite{NPDIT} is
\begin{equation}\label{eq:Tikhonov_like}
    P_n = A^TA + \nu_n I_d, \quad \forall \ n\geq 0,
\end{equation}
where \( I_d \in \mathbb{R}^{d\times d} \) is the identity matrix and \( \{\nu_n\}_{n\in\mathbb{N}} \) is a sequence of positive numbers such that $P_n$ is positive definite. This choice for the preconditioner draws from regularizing preconditioning techniques used in image deblurring, particularly the Iterated Tikhonov method \cite{Huang-etal-2013, donatelli2013fast, cai2016regularization}, and for this reason, the algorithm was called \textit{Nested Primal-Dual Iterated Tikhonov} (NPDIT). Although this algorithm shows excellent performance when applied to problem \eqref{eq:TV}, it suffers from high computational demand when high accuracy is required for approximating the proximity operator. In this work, we address this issue by noting that a variable metric approach reduces to a right preconditioning strategy when considering optimization problems of the form~\eqref{eq:problem_leastsquares}. Iterative schemes like \eqref{eq:pg_VM} are based on left preconditioning strategies, resulting in faster algorithms since the scaling matrix, or equivalently, the preconditioner \( P_n \), does not affect the dual sequence used in the nested iteration. Our numerical results also demonstrate that the reconstruction quality is comparable to that achieved by the NPDIT algorithm. Moreover, we show that a sufficiently accurate proximal evaluation is necessary if the shifting parameter \( \nu_n \) is small or if one aims to eliminate the extrapolation step. In this regard, with appropriate choices of the non-smooth term \( h\circ W \), the preconditioner \( P_n \), and the extrapolation parameter \( \gamma_n \), the scheme~\eqref{eq:pg_VM} aligns with standard algorithms such as ISTA, FISTA, and ITTA \cite{Daubechies-et-al-2004, Beck-Teboulle-2009b, Huang-etal-2013}.

This paper is organized as follows: Section \ref{sec:preliminaries} introduces the basic notions and results of convex analysis, which are used in Section \ref{sec:primal--dual} to define and analyze a primal-dual approach and a variable metric framework. In Section \ref{eq:secright}, we demonstrate that employing the right preconditioning strategy for problem~\eqref{eq:problem_leastsquares} is equivalent to adopting a variable metric approach, provided that the preconditioner \( P_n \) is chosen appropriately. Section \ref{ch:PNPD} presents our proposal, based on the left preconditioning approach, along with convergence results and possible strategies for selecting a suitable preconditioner \( P_n \). Section \ref{sec:numerical_results} is centered on the numerical tests, especially comparing our proposal with the NPDIT method. Finally, Section \ref{sec:concl} is devoted to conclusions.

\section{Preliminaries}\label{sec:preliminaries}
This section aims to recall some well-known definitions and properties of convex analysis that will be exploited throughout the paper.
Before that, we introduce some useful notations that will be used consistently.

Denoting by 
$\mathcal{S}_+(\R^d)$ the set of all $d\times d$ real symmetric positive definite matrices and given $\eta>0$, we define the subset $\mathcal{D}_{\eta}\subseteq \mathcal{S}_+(\R^{d})$ as containing all matrices in $\mathcal{S}_+(\R^{d})$ with eigenvalues greater than or equal to $\eta$. Moreover, given $M\in \mathcal{S}_+(\R^d)$, the norm induced by $M$ is defined as
$\|u\|_{M} = \sqrt{u^TMu}$, $\forall u\in\R^d$.

For a function $f:\R^d\rightarrow\R\cup\{\infty\}$, we define the domain of $f$ as the set $\operatorname{dom}(f) := \{x\in\R^d : f(x)<\infty\}$.
Finally, we recall that the relative interior of $\Omega\subseteq \R^n$ is the set
$ \operatorname{relint}(\Omega) = \{u\in\Omega:  \exists \ \epsilon>0 \ \text{s.t.} \ B(u,\epsilon)\cap \operatorname{aff}(\Omega)\subseteq \Omega\}$,
where $B(u,\epsilon)$ is the ball centered at $u$ with radius $\epsilon$, and $\operatorname{aff}(\Omega)$ is the affine hull of $\Omega$.

\begin{definition}
    A function $f:\R^d\rightarrow \R$ is said to have a $L-$Lipschitz continuous gradient if the following property holds
    \begin{equation*}
        \|\nabla f(\tilde{u})-\nabla f(\bar{u})\|\leq L\|\tilde{u}-\bar{u}\|, \quad \forall \ \tilde{u},\bar{u}\in\R^d.
    \end{equation*}
\end{definition}

\begin{definition}
    Let $\varphi:\R^d\rightarrow\R\cup\{\infty\}$ be a proper convex and lower semicontinuous function. The subdifferential of $\varphi$ at point $u\in\R^d$ is defined as the~set
    \begin{equation*}
        \partial\varphi = \{w\in\R^d: \varphi(v)\ge\varphi(u) + \left \langle w , v-u\right \rangle, \quad \forall v\in\R^d\}.
    \end{equation*}
\end{definition}

\begin{lemma}
    Let $\varphi_1,\varphi_2 : \R^{d'}\rightarrow\R\cup\{\infty\}$ be proper, convex and lower semicontinuous functions and $W\in\R^{d'\times d}$.
    \begin{itemize}
        \item If $\varphi(v) = \varphi_1(v)+\varphi_2(v)$ and there exists $v_0\in\R^{d'}$ such that $v_0\in \operatorname{relint}(\operatorname{dom}(\varphi_1))\cap\operatorname{relint}(\operatorname{dom}(\varphi_2))$, then
              \begin{equation*}
                  \partial(\varphi_1+\varphi_2)(v) = \partial\varphi_1(v)+\partial\varphi_2(v), \quad \forall v\in\operatorname{dom}(\varphi).
              \end{equation*}
        \item If $\varphi = \varphi_1(Wu)$ and there exist $u_0\in\R^d$ such that $Au_0\in\operatorname{relint}(\operatorname{dom}(\varphi_1))$, then
              \begin{equation*}
                  \partial\varphi(u) = W^T\partial\varphi_1(Au) =\{W^Tv : v\in\partial\varphi_1(Au)\}, \quad \forall u\in\operatorname{dom}(\varphi).
              \end{equation*}
    \end{itemize}
\end{lemma}

\begin{definition}\cite[p. 104]{Rockafellar70}
    Given a proper, convex, lower semicontinuous function $\varphi:\mathbb{R}^d\rightarrow \mathbb{R}\cup\{\infty\}$, the convex conjugate of $\varphi$ is the function
    \begin{equation*}
        \varphi^*:\R^d\rightarrow \R\cup\{\infty\}, \quad \varphi^*(w)=\underset{u\in\R^d}{\operatorname{\sup}}\ \langle w,u\rangle - \varphi(u), \ \forall \ w\in\R^d.
    \end{equation*}
\end{definition}
The following well-known result holds for the biconjugate function $(\varphi^*)^*$.
\begin{lemma}\cite[Theorem 12.2]{Rockafellar70}\label{lem:biconjugate}
    Let $\varphi:\R^d\rightarrow \R\cup\{\infty\}$ be proper, convex, and lower semicontinuous. Then $\varphi^*$ is convex and lower semicontinuous and $(\varphi^*)^*=\varphi$, namely,
    \begin{equation*}
        \varphi(u)=\underset{w\in\R^d}{\operatorname{\sup}}\ \langle u,w\rangle - \varphi^*(w), \quad \forall \ u\in\R^d.
    \end{equation*}
\end{lemma}

\begin{definition}\label{def:proxP}\cite[p. 278]{Moreau65}, \cite[p. 877]{Frankel-etal-2015}
    The proximal operator of a proper, convex, lower semicontinuous function $\varphi$ with parameter $\gamma>0$ is the map $\mathrm{prox}_{\gamma \varphi}:\mathbb{R}^d\rightarrow\mathbb{R}^d$ defined as
    $$
        \mathrm{prox}_{\gamma \varphi}(a) = \arg\min_{u\in\mathbb{R}^d}\varphi(u)+\frac{1}{2\gamma}\|u-a\|^2, \quad \forall \ a\in\mathbb{R}^d.
    $$
    Similarly, we define the proximal operator with respect to the norm induced by $P\in\mathcal{D}_{\eta}$ as
    \begin{equation*}
        \prox_{\gamma \varphi}^P(a) = \argmin_{u\in\R^d}\varphi(u) +\frac{1}{2\gamma}\|u-a\|_P^2, \quad \forall a\in\R^d.
    \end{equation*}
\end{definition}

\begin{lemma}
    Let $\varphi:\R^d\rightarrow\R\cup\{\infty\}$ be proper, convex, and lower semicontinuous. For all $\alpha,\beta>0$, the following statements are equivalent:
    \begin{enumerate}
        \item u = $\prox_{\alpha\varphi}(u+\alpha w)$;
        \item $w\in\partial\varphi(u)$;
        \item $\varphi(u)+\varphi^*(w) = \left \langle w, u\right \rangle$;
        \item $u\in\partial\varphi^*(w)$;
        \item $w = \prox_{\beta\varphi^*}(\beta u+w)$.
    \end{enumerate}
\end{lemma}

Using Lemma \ref{lem:biconjugate}, we can reformulate the initial problem \eqref{eq:problem_formulation} as the following convex-concave saddle-point problem
\begin{equation}\label{eq:primal_dual}
    \min_{u\in\R^d}\max_{v\in\R^{d'}} \ \mL(u,v)\equiv f(u)+\langle Wu,v\rangle -h^*(v),
\end{equation}
where $\mL(u,v)$ denotes the primal--dual function. A solution of \eqref{eq:primal_dual} is any point $(\hat{u},\hat{v})\in\R^d\times \R^{d'}$ such that
\begin{equation}\label{eq:saddle}
    \mL(\hat{u},v)\leq \mL(\hat{u},\hat{v})\leq \mL(u,\hat{v}), \quad \forall \ u\in\R^d, \ \forall \ v\in\R^{d'}.
\end{equation}

From now on, we consider the optimization problem \eqref{eq:problem_formulation} under the following assumptions.

\begin{assumption}
    \label{hyp: standard}
    \quad
    \begin{enumerate}
        \item $f: \mathbb{R}^d \to \mathbb{R}$ is convex and differentiable with $L$-Lipschitz continuous gradient;
        \item \label{hyp: standard_h} $h: \mathbb{R}^{d'} \to \overline{\mathbb{R}}$ is a proper convex lower semicontinuous function;
        \item \label{hyp: standard_W} $W\in \mathbb{R}^{d'\times d}$ and exists $u_0$ such that $Wu_0 \in \relint(\dom(h))$;
        \item Problem in equation \eqref{eq:problem_formulation} has at least one solution (the set in equation \eqref{eq:problem_formulation} is not empty).
    \end{enumerate}
\end{assumption}

We remark that the assumption on $W$ is needed to guarantee that the subdifferential rule $\partial (h \circ W)(u)=W^T\partial h(Wu)$ holds, thus we can interpret the minimum points of \eqref{eq:problem_formulation} as solutions of appropriate variational equations, as stated below.

\begin{lemma}\cite[Lemma 3.1]{Chen-et-al-2018}\label{lem:tech0}
    Under Assumption \ref{hyp: standard}, a point $\hat{u}\in\R^d$ is a solution of problem \eqref{eq:problem_formulation} if and only if the following conditions hold
    \begin{equation}\label{eq:minimizer_key}
        \begin{cases}
            \nabla f(\hat{u})+W^T\hat{v}=0, \\
            \hat{v}=\prox_{\beta\alpha^{-1}h^*}(\hat{v}+\beta\alpha^{-1}W\hat{u}),
        \end{cases}\quad \forall \ \alpha,\beta>0.
    \end{equation}
\end{lemma}

\section{Inexact proximal gradient methods}\label{sec:primal--dual}
In this section, we will revisit a class of methods used to solve the initial problem \eqref{eq:problem_formulation} when the proximal operator of \( h \circ W \) does not have a closed-form solution or is computationally expensive. To address this issue, we will report results that provide an approach to approximate the proximal operator using a primal-dual scheme. Additionally, we will provide convergence results for each method towards the solution of our initial problem \eqref{eq:problem_formulation}.

\subsection{Nested Primal-Dual Method (NPD)}
This method can be seen as an inexact inertial forward-backward algorithm, and its iterative scheme can be summarized as follows:
\begin{align}
    \label{eq:NPD}
    \begin{cases}
        \bar{u}_n = u_n + \gamma_n (u_n - u_{n-1}), \\
        u_{n+1} \approx \prox_{\alpha_n h \circ W}(\bar{u}_n - \alpha_n \nabla f(\bar{u}_n)),
    \end{cases}
\end{align}
where \( \approx \) indicates an approximation of the proximal-gradient point. Here, \( \bar{u}_n \) is usually referred to as the inertial point, \( \gamma_n \in [0,1] \) is the inertial parameter, and $ \alpha_n>0$ is the step length along the descent direction \( -\nabla f(\bar{u}_n) \). Since in the iterative scheme \eqref{eq:NPD} we only need an approximation of the proximity operator of the non-smooth term in \eqref{eq:problem_formulation}, we state the following result.

\begin{lemma}\label{lemma:primal_dual_NPD}
    Suppose that \( h: \mathbb{R}^{d'} \rightarrow \mathbb{R} \cup \{\infty\} \) and \( W \in \mathbb{R}^{d' \times d} \) satisfy Assumption \ref{hyp: standard} (ii)-(iii). Given \( a \in \mathbb{R}^d \), \( \alpha > 0 \), \( 0 < \beta < 2/\|W\|^2 \), and \( v^0 \in \mathbb{R}^{d'} \), define the sequence
    \begin{equation}\label{eq:dual_seq_NPD}
        v^{k+1} = \prox_{\beta \alpha^{-1}h^*}(v^k+\beta\alpha^{-1}W(a-\alpha W^Tv^k)), \quad \forall k \geq 0.
    \end{equation}
    Then there exists \(\hat{v}\in\mathbb{R}^{d'}\) such that:
    \begin{itemize}
        \item[(i)] \(\lim_{k\rightarrow\infty}v^{k} = \hat{v}\);
        \item[(ii)] \(\prox_{\alpha h\circ W}(a) = a - \alpha W^T\hat{v}.\)
    \end{itemize}
\end{lemma}

\begin{proof}
    See \cite{Bonettini-et-al-2023a}, Theorem 1.
\end{proof}

Note that by combining items $(i)$ and $(ii)$ of Lemma \ref{lemma:primal_dual_NPD} with equation \eqref{eq:dual_seq_NPD}, it holds
\begin{equation*}
    \prox_{\alpha h\circ W}(a) = \lim_{k\rightarrow\infty}a-\alpha W^T\prox_{\beta \alpha^{-1}h^*}(v^k+\beta\alpha^{-1}W(a-\alpha W^Tv^k)).
\end{equation*}
In this way, we can compute an approximation of the proximal operator of \( h\circ W \) by a finite number of steps of a primal-dual procedure, provided that the operator \(\prox_{\beta\alpha^{-1} h^*}\) and the matrix-vector product with the linear operators \( W \) and \( W^T \) are easily computable. Let \( k_{\mathrm{max}} \) be the maximum number of inner iterations used to compute the dual sequence \eqref{eq:dual_seq_NPD}, and recalling that \( L \) is the Lipschitz constant of the gradient of \( f \), the resulting NPD algorithm is summarized in Algorithm \ref{alg:npd}.

\begin{algorithm}
    \caption{NPD}
    \label{alg:npd}
    \begin{algorithmic}[1]
        \State Choose \( u_0 \in \mathbb{R}^d \), \( 0<\alpha<\frac{1}{L} \), \( 0<\beta<\frac{1}{\|W\|^2} \), \( k_{\mathrm{max}} \in \mathbb{N} \), \(\{\gamma_n\}_{n\in \mathbb{N}} \subseteq \mathbb{R}_{\ge 0}\).
        \State \( u_{-1} = u_0 \)
        \State \( u_{0}^0 = 0 \)
        \For{\( n = 0, 1, \dots \)}
        \State \(\bar{u}_n = u_n + \gamma_n (u_n - u_{n-1})\)
        \State \( u_{n+\frac{1}{2}} = \bar{u}_n - \alpha \nabla f(\bar{u}_n) \)
        \For{\( k = 0, 1, \dots, k_{\text{max}}-1 \)}
        \State \( u_{n}^{k} = u_{n+\frac{1}{2}} - \alpha W^T v_{n}^k \)
        \State \( v_{n}^{k+1} = \prox_{\beta \alpha^{-1} h^*}(v_{n}^k + \beta \alpha^{-1} W u_{n}^{k}) \)
        \EndFor
        \State \( v_{n+1}^0 = v_{n}^{k_{\text{max}}} \)
        \State \( u_{n}^{k_{\text{max}}} = u_{n+\frac{1}{2}} - \alpha W^T v_{n}^{k_{\text{max}}} \)
        \State \( u_{n+1} = \frac{1}{k_{\text{max}}} \sum_{j=1}^{k_{\text{max}}} u_{n}^{j} \)
        \EndFor
    \end{algorithmic}
\end{algorithm}

In Algorithm \ref{alg:npd}, a ``warm-up strategy'' is also considered, meaning that each inner primal-dual loop is ``warm started'' with the outcome of the previous one. In \cite{Bonettini-et-al-2023a}, this has been proven to be sufficient to show the convergence of the iterates to a solution of \eqref{eq:problem_formulation} when the accuracy in the proximal evaluation is preset.

To complete this brief description of the NPD method, we recall the convergence result.

\begin{theorem}[Convergence of NPD]
    \label{thm:convergence_npd}
    Suppose that $f, h$, and $W$ satisfy Assumption \ref{hyp: standard}. Let \( \{(u_n, v_n^0)\}_{n\in \mathbb{N}} \) be the primal-dual sequence generated by the NPD Algorithm \ref{alg:npd} with \( \alpha_n = \alpha \in (0, \frac{1}{L}] \) and \( \beta_n = \beta \in (0, \|W\|^{-2}) \) for all \( n \in \mathbb{N} \). Suppose also that the inertial parameter \( \{\gamma_n\}_{n\in\mathbb{N}} \) satisfies
    \begin{equation}\label{eq:conv_NPD_gamma}
        \sum_{n=0}^\infty\gamma_n\|u_n-u_{n-1}\|<\infty.
    \end{equation}
    Then, the following statements hold:
    \begin{itemize}
        \item[(i)] the sequence \( \{(u_n,v_n^0)\}_{n\in\mathbb{N}} \) is bounded;
        \item[(ii)] the sequence \( (\{u_n, v_n^0)\}_{n\in \mathbb{N}} \) converges to a solution of \eqref{eq:primal_dual} and therefore the primal sequence \( \{u_n\}_{n\in \mathbb{N}} \) converges to a solution of the initial problem \eqref{eq:problem_formulation}.
    \end{itemize}
\end{theorem}

\begin{proof}
    See \cite{Bonettini-et-al-2023a}, Theorem 2.
\end{proof}

\subsection{A variable metric nested primal--dual framework}

We introduce a more sophisticated variant of the iteration \eqref{eq:NPD}, proposed in \cite{NPDIT}, which is a variable metric strategy combined with inexact proximal evaluations to achieve faster convergence to the minimum point of the problem \eqref{eq:problem_formulation}.

In general, given $\eta>0$ and \( P_n \in \mathcal{D}_{\eta} \), an inertial variable metric forward-backward method for problem \eqref{eq:problem_formulation} can be expressed as
\begin{equation}\label{eq:FB}
    \begin{cases}
        \bar{u}_n = u_n + \gamma_n(u_n-u_{n-1}), \\
        u_{n+1} \approx \mathrm{prox}^{P_n}_{\alpha_n h \circ W}(\bar{u}_n - \alpha_n P_{n}^{-1}\nabla f(\bar{u}_n)),
    \end{cases}
\end{equation}
where \(\mathrm{prox}^{P_n}_{\alpha_n h \circ W}\) is defined following Definition \ref{def:proxP}. The proximal operator with a metric defined by \( P_n \in \mathcal{D}_{\eta} \) can be approximated by an appropriate sequence of primal-dual iterates, similar to the case in NPD. The following result generalizes the one derived in Lemma \ref{lemma:primal_dual_NPD}, taking into account the presence of the scaling matrix  \( P_n\).

\begin{lemma}\label{lem:primal_dual}
    Suppose that \( h:\mathbb{R}^{d'}\rightarrow \mathbb{R}\cup\{\infty\} \) and \( W \in \mathbb{R}^{d'\times d} \) satisfy Assumptions~\ref{hyp: standard} (ii)-(iii). Given \( a \in \mathbb{R}^d \), \( \alpha > 0 \), \( 0 < \beta < \frac{2}{\|WP^{-1}W^T\|} \), and \( v^0 \in \mathbb{R}^{d'} \), define the sequence
    \begin{equation}\label{eq:primal--dual-procedure}
        v^{k+1} = \prox_{\beta \alpha^{-1}h^*}(v^k + \beta\alpha^{-1}W(a - \alpha P^{-1}W^Tv^k)), \quad \forall k \geq 0.
    \end{equation}
    Then there exists \( \hat{v} \in \mathbb{R}^{d'} \) such that:
    \begin{itemize}
        \item[(i)] \(\lim_{k\rightarrow\infty} v^{k} = \hat{v}\),
        \item[(ii)] \(\prox_{\alpha h \circ W}^P(a) = a - \alpha P^{-1}W^T\hat{v}\).
    \end{itemize}
\end{lemma}

\begin{proof}
    See \cite{NPDIT}, Lemma 5.
\end{proof}

Algorithm \ref{alg:npdit} provides the pseudocode for the variable metric approach. The introduction of the scaling matrix \( P_n \in \mathcal{D}_{\eta} \) influences the choices of the step length \( \alpha \) and the dual variable \( \beta \). Specifically, given \( \epsilon, \delta \in (0,1) \), we set \( \beta_n = \epsilon / (\|P_n^{-1}\|\|W\|^2) \) at each step, while a backtracking strategy is used to determine the step size \( \alpha_n \). Starting with an initial approximation \( L_{-1} > 0 \) of the Lipschitz constant of \( \nabla f \), according to \cite{NPDIT}, we set \( L_n = L_{n-1} \) and \( \alpha_n = \epsilon / L_n \). We then increase \( L_n \) by a factor of \( 1/\delta \) until the following condition is satisfied:
\begin{equation}
    \label{eq:npdit_backtracking_condition}
    f(u_{n+1}) \leq f(\bar{u}_n) + \nabla f(\bar{u}_n)^T(u_{n+1} - \bar{u}_n) + \frac{L_n}{2}\|u_{n+1} - \bar{u}_n\|_{P_n}^2.
\end{equation}
As noted in \cite{NPDIT}, inequality \eqref{eq:npdit_backtracking_condition} holds for \( L_n \ge L/\eta \), where \( L \) is the Lipschitz constant of \( \nabla f \). Thus, the backtracking procedure concludes in a finite number of steps.

Before concluding this discussion on variable metric and inexact proximal gradient methods, we present a convergence result. Before stating it, we need to impose an additional condition on the scaling matrix sequence \( \{P_n\}_{n\in\mathbb{N}} \) that is
\begin{assumption}\label{hyp:Pn}
    The sequence of scaling matrices \( \{P_n\}_{n\in\mathbb{N}} \) in Algorithm \ref{alg:npdit} is chosen such that
    \begin{equation}\label{eq:cond_Pn}
        P_{n}\preceq (1+\zeta_{n-1})P_{n-1}, \  \forall \ n\geq 0, \quad \text{where } \zeta_{n-1} \geq 0, \ \sum_{n=0}^\infty\zeta_{n-1}<\infty.
    \end{equation}
\end{assumption}

\begin{theorem}[\cite{NPDIT}]\label{thm:convergenceNPDIT}
    Suppose Assumption \ref{hyp: standard} holds, and let \((\hat{u},\hat{v}) \in \mathbb{R}^d \times \mathbb{R}^{d'}\) be a solution of the primal-dual problem \eqref{eq:primal_dual}. Let \(\{(u_n,v_n^0)\}_{n\in\mathbb{N}}\) be the primal-dual sequence generated by Algorithm \ref{alg:npdit}. If the inertial parameters \(\{\gamma_n\}_{n\in\mathbb{N}}\) satisfy the condition \eqref{eq:conv_NPD_gamma}, and the scaling matrices \(\{P_n\}_{n\in\mathbb{N}}\) in Algorithm \ref{alg:npdit} satisfy Assumption \ref{hyp:Pn}, then the following statements hold:
    \begin{itemize}
        \item[(i)] The sequence \(\{(u_n,v_n^0)\}_{n\in\mathbb{N}}\) is bounded.
        \item[(ii)] Defining $ D_n = I_{d'}- \beta_n W P_{n}^{-1}W^T$ and given a saddle point \((\hat{u},\hat{v})\) solution of \eqref{eq:primal_dual}, the sequence \(\{\beta_{n-1} k_{\max}\|\hat{u} - u_n\|_{P_{n-1}}^2 + \alpha_{n-1}^2 \|\hat{v} - v_n^0\|^2_{D_{n-1}}\}_{n\in\mathbb{N}}\) converges.
        \item[(iii)] The sequence \(\{(u_n,v_n^0)\}_{n\in\mathbb{N}}\) converges to a solution of \eqref{eq:primal_dual}.
    \end{itemize}
\end{theorem}

\begin{algorithm}
    \caption{Nested Primal-Dual Variable Metric method}
    \label{alg:npdit}
    \begin{algorithmic}[1]
        \State Choose \( u_0 \in \mathbb{R}^d \), \( u_{-1} = u_0 \), \( v_{-1}^{k_{\max}} \in \mathbb{R}^{d'} \), \( \epsilon \in (0,1) \), \( \delta \in (0,1) \), \( L_{-1} > 0 \), \( \eta > 0 \), \( \alpha_{-1} = \frac{\epsilon}{L_{-1}} \), \( P_{-1} \in \mathcal{D}_{\eta} \), \( \beta_{-1} = \frac{\epsilon}{\|P_{-1}^{-1}\|\|W\|^2} \)
        \For{\( n = 0, 1, \dots \)}
        \State Choose \( \gamma_n \geq 0 \) and compute the extrapolated point \( \bar{u}_n = u_n + \gamma_n (u_n - u_{n-1}) \)
        \State Choose \( P_n \in \mathcal{D}_{\eta} \)
        \State Set \( i_n = 0 \), \( L_n = L_{n-1} \), \( \alpha_n = \frac{\epsilon}{L_n} \), \( \beta_n = \frac{\epsilon}{\|P_n^{-1}\|\|W\|^2} \), \( v_n^0 = v_{n-1}^{k_{\max}} \)
        \For{\( k = 0, 1, \dots, k_{\max} - 1 \)}
        \State \( u_n^k = \bar{u}_n - \alpha_n P_n^{-1} \left(\nabla f(\bar{u}_n) + W^T v_n^k\right) \)        
        \State \( v_n^{k+1} = \prox_{\beta_n \alpha_n^{-1} h^*}(v_n^k + \beta_n \alpha_n^{-1} W u_n^k) \)
        \EndFor
        \State \( u_n^{k_{\max}} = \bar{u}_n - \alpha_n P_n^{-1} \nabla f(\bar{u}_n) - \alpha_n P_n^{-1} W^T v_n^{k_{\max}} \)
        \State \( \tilde{u}_n = \frac{1}{k_{\max}} \sum_{k=1}^{k_{\max}} u_n^k \)
        \If{\( f(\tilde{u}_n) \leq f(\bar{u}_n) + \nabla f(\bar{u}_n)^T (\tilde{u}_n - \bar{u}_n) + \frac{L_n}{2} \|\tilde{u}_n - \bar{u}_n\|^2_{P_n} \)}
        \State \( u_{n+1} = \tilde{u}_n \)
        \Else
        \State \( i_n = i_n + 1 \), \( L_n = \frac{L_{n-1}}{\delta^{i_n}} \), \( \alpha_n = \frac{\epsilon}{L_n} \)
        \State Go to Step 7
        \EndIf
        \EndFor
    \end{algorithmic}
\end{algorithm}

\section{Variable metric approach as right preconditioning}\label{eq:secright}
In this section, we give a different interpretation of the variable metric approach \eqref{eq:FB}, in the case of a stationary preconditioner $P_n=P$, when applied to the model problem \eqref{eq:problem_leastsquares}.
We will prove that the variable metric approach can be reformulated as a right preconditioning strategy for NPD.

Considering image deblurring problems, we reduce our analysis to the case of $f(u)=\frac{1}{2}\|Au - b^{\delta}\|^2$, i.e., to problem \eqref{eq:problem_leastsquares}.
Let $R \in \mathbb{R}^{d \times d}$ be invertible,
applying the right preconditioning in \eqref{eq:problem_leastsquares}, we obtain the equivalent formulation
\begin{equation}
    \label{eq:NPD_problem_right_preconditioner}
    \hat{u} =
    \argmin_{u \in \mathbb{R}^d} \frac{1}{2}\|AR^{-1}Ru - b^{\delta}\|^2 + h(Wu).
\end{equation}
Setting $\tilde{A} = AR^{-1}$, $z = Ru$, and $\tilde{W} = WR^{-1}$, the solution $\hat{u}$ of \eqref{eq:problem_leastsquares} is given by
\begin{equation*}
    \hat{u} = R^{-1} \hat{z},
\end{equation*}
where
\begin{equation}\label{eq:ls-hat}
    \hat{z} = \argmin_{z \in \mathbb{R}^d} \frac{1}{2}\|\tilde{A} z - b^{\delta}\|^2 + h(\tilde{W}z) = \argmin_{z \in \mathbb{R}^d} \tilde{f}(z) + h(\tilde{W}z).
\end{equation}
The NPD method applied to \eqref{eq:ls-hat} results in
\begin{equation} \label{eq:right_preconditioner_NPD_scheme}
    \begin{cases}
        \bar{z}_n = z_n + \gamma_n(z_n - z_{n-1}), \\
        z_{n+1} = \prox_{\alpha_n h \circ \tilde{W}}(\bar{z}_n - \alpha_n \nabla \tilde{f}(\bar{z}_n)),
    \end{cases}
\end{equation}
where, for simplicity of notation, the ``$\approx$'' symbol has been replaced with the ``$=$'' symbol
assuming that the proximity operator of the non-differentiable part can be computed exactly.
In the end, we will revert to the approximation notation
considering the same approximation scheme as the inexact variable case.

Let $u_n = R^{-1} z_n$, clearly
$\lim_n z_n = \hat z$ implies $\lim_n u_n = \hat u$.
Therefore, we want to write the scheme in equation \eqref{eq:right_preconditioner_NPD_scheme} in terms of the sequence $\{u_n\}_{\in \mathbb{N}}$.

\begin{lemma}\label{lemma:proxR}
    Let $h$ be a convex lower semicontinuous function. Let
    $W\in \R^{d'\times d}$, $R\in \R^{d\times d}$ invertible, and $\tilde W = WR^{-1}$, then
    \begin{equation*}
        \prox_{\alpha h \circ \tilde W}(z) = R \prox_{\alpha h \circ W}^{R^TR}(R^{-1}z).
    \end{equation*}
\end{lemma}

\begin{proof}
    We observe that, for a function $g:\R^d\to\R$, it holds
    \[
        \hat z \in \argmin_{\tilde z} \ g(R^{-1}\tilde z)\ \Leftrightarrow\ 
        R^{-1} \hat z \in \argmin_{\tilde u} \ g(\tilde u)\ \Leftrightarrow\ 
        \hat z \in R \argmin_{\tilde u} \ g(\tilde u),
    \]
    Therefore, by fixing $u = R^{-1} z$, we have
    \begin{align*}
        \prox_{\alpha h \circ \tilde W}(z) & = \argmin_{\tilde z} \ h(WR^{-1}\tilde z) + \frac{1}{2\alpha} \|z-\tilde z\|^2 \\
        & = R \argmin_{\tilde u} \ h(W \tilde u) + \frac{1}{2\alpha} \|R(u-\tilde u)\|^2
        \\ & = R
        \prox_{\alpha h \circ W}^{R^TR}(R^{-1}z).
    \end{align*}
\end{proof}

Recalling that $f(u)=\frac{1}{2}\|Au - b^{\delta}\|^2$ and $\tilde f(u)=\frac{1}{2}\|\tilde Au - b^{\delta}\|^2$, it holds
\begin{equation}\label{eq:tildef}
    \nabla \tilde f(Rx) = R^{-T} \nabla f(x).
\end{equation}
Therefore, using Lemma \ref{lemma:proxR} and equation \eqref{eq:tildef}, from the second equation in \eqref{eq:right_preconditioner_NPD_scheme}, for $\bar u = R^{-1} \bar z$ we have
\begin{align*}
    u_{n+1} & = R^{-1} z_{n+1}                                                                             \\
            & = \prox_{\alpha_n h \circ W}^{R^TR}(R^{-1} \bar z_n - \alpha_n R^{-1} \nabla \tilde f(\bar z_n)) \\
            & = \prox_{\alpha_n h \circ W}^{R^TR}(\bar u_n - \alpha_n (R^T R)^{-1} \nabla f(\bar u_n)).
\end{align*}

Finally, adding also the extrapolation step and replacing back the approximation symbol in the second equation of \eqref{eq:right_preconditioner_NPD_scheme}, the NPD method applied to problem \eqref{eq:NPD_problem_right_preconditioner} is given by the iterative scheme
\begin{equation}
    \label{eq:right_preconditioner_FBS_scheme_explicit}
    \begin{cases}
        \bar{u}_n= u_n + \gamma_n(u_n-u_{n-1}), \\
        u_{n+1} \approx \prox_{\alpha_n h \circ W}^{R^TR}(\bar u_n - \alpha_n (R^T R)^{-1} \nabla f(\bar u_n)).
    \end{cases}
\end{equation}
Note that, given a stationary preconditioner $P_n = P \in \mathcal{D}_{\eta}$, by choosing $R=P^{\frac{1}{2}}$ such that $R^T R = P$, then the iteration \eqref{eq:right_preconditioner_FBS_scheme_explicit} is exactly the variable metric scheme \eqref{eq:FB}.

\section{Preconditioned Nested Primal-Dual (PNPD)}
\sectionmark{Preconditioned Nested Primal-Dual}
\label{ch:PNPD}
In Section~\ref{eq:secright}, we discussed how NPDIT applied to problem \eqref{eq:problem_leastsquares} can be interpreted as a right preconditioning strategy. Hence, it is natural to explore the effects of the left preconditioning strategy instead of the right one.
The seminal idea is to accelerate the convergence of the NPD method replacing the gradient descent step inside the proximal evaluation with a preconditioned iteration.
Focusing on the image deblurring problem, we want to define an iterative scheme faster than NPD, potentially requiring less computational effort at each iteration than NPDIT. We achieved this by modifying the data fidelity norm with the one induced by an appropriate preconditioner $S\in \mathcal{S}_+(\mathbb{R}^s)$.

The first part of this section will be devoted to deriving and analyzing our proposal and discussing possible strategies for choosing the preconditioner~$S$. Moreover, we will show how with proper choices of the preconditioner and the non-smooth term $h \circ W$, our method coincides with standard proximal gradient approaches such as FISTA and ITTA. A possible drawback of our approach is that it changes the norm of the data fidelity term computing a solution potentially different from the solution of \eqref{eq:problem_leastsquares}. This issue will be addressed in the final part of this section, where we introduce non-stationarity in the preconditioner letting the sequence of preconditioners $S_n$ converge to the identity operator. This will allow us to compare our proposal with standard approaches that use the $\ell^2$-norm in the data fidelity term. 

Similarly to equation \eqref{eq:problem_leastsquares}, 
given $S\in \mathcal{S}_+(\mathbb{R}^s)$, we consider the optimization problem
\begin{align}
    \label{eq:model_PNPD}
    \begin{split}
        \argmin_{u\in\mathbb{R}^d} f_S(u) + h(Wu) = \argmin_{u\in\mathbb{R}^d}\frac{1}{2} \|Au-b^{\delta}\|_{S^{-1}}^2 + h(Wu).
    \end{split}
\end{align}
In the data fidelity term
\begin{equation*}
    f_S(u) = \frac{1}{2}\|S^{-\frac{1}{2}}(Au-b^{\delta})\|^2,
\end{equation*}
the linear operator $S^{\frac{1}{2}}$ can be interpreted as a left preconditioner for the linear system \eqref{eq:Model_Equation}. If we further assume that there exists $P\in \mathcal{S}_+(\R^d)$ such that
\begin{equation}\label{eq:commutative_PNPD}
    P^{-1}A^T = A^T S^{-1},
\end{equation}
then it holds
\begin{equation}\label{eq:GfS}
    \nabla f_S(u) = A^TS^{-1}(Au-b^{\delta}) = P^{-1}\nabla f(u),
\end{equation}
where $f(u) = \frac{1}{2}\|Au-b^{\delta}\|^2$ as in equation \eqref{eq:problem_leastsquares}.

Therefore, applying the NPD algorithm to problem~\eqref{eq:model_PNPD}, under the assumption~\eqref{eq:commutative_PNPD}, we obtain a two-step iterative scheme, named \textit{Preconditioned Nested Primal-Dual} (PNPD), defined as
\begin{equation}
    \label{eq:PNPD_scheme}
    \begin{cases}
        \bar{u}_n= u_n + \gamma_n(u_n-u_{n-1}), \\
        u_{n+1} \approx \prox_{\alpha_n h \circ W}(\bar u_n - \alpha_n P^{-1} \nabla f(\bar u_n)).
    \end{cases}
\end{equation}
The main difference between the PNPD method and NPDIT is that we are using the
standard definition of proximity operator of the non--differentiable part
$h\circ W$. From this, it follows that the scaling matrix $P$ no longer affects
the dual variable in the nested iteration and neither the dual step length
$\beta_n$ that now relies only on the norm of the operator $W^TW$ as in NPD. 
 This implies that in PNPD the preconditioner acts only on the primal part, while in NPDIT it affects also the dual part.
It follows that, as the inner iterations increase, the computational cost of PNPD
will become lower compared to that of NPDIT.

Algorithm \ref{alg:pnpd} reports the pseudocode of the proposed PNPD method.

\begin{algorithm}
    \caption{PNPD}
    \label{alg:pnpd}
    \begin{algorithmic}[1]
        \State Choose \( u_0 \in \mathbb{R}^d \), \(
        P\in\mathcal{S}_+(\R^d) \), \( 0<\alpha<\frac{1}{L_S} \), \(
        0<\beta<\frac{1}{\|W\|^2} \), \( k_{\mathrm{max}} \in \mathbb{N} \),
        \(\{\gamma_n\}_{n\in \mathbb{N}} \subseteq \mathbb{R}_{\ge 0}\).
        \State \( u_{-1} = u_0 \)
        \State \( v_{0}^0 = 0 \)
        \For{\( i = 0, 1, \dots \)}
        \State \(\bar{u}_n = u_n + \gamma_n (u_n - u_{n-1})\)
        \State \( u_{n+\frac{1}{2}} = \bar{u}_n - \alpha P^{-1}\nabla f(\bar{u}_n) \)
        \For{\( k = 0, 1, \dots, k_{\text{max}}-1 \)}
        \State \( u_{n}^{k} = u_{n+\frac{1}{2}} - \alpha W^T v_{n}^k \)
        \State \( v_{n}^{k+1} = \prox_{\beta \alpha^{-1} h^*}(v_{n}^k + \beta \alpha^{-1} W u_{n}^{k}) \)
        \EndFor
        \State \( v_{n+1}^0 = v_{n}^{k_{\text{max}}} \)
        \State \( u_{n}^{k_{\text{max}}} = u_{n+\frac{1}{2}} - \alpha W^T v_{n}^{k_{\text{max}}} \)
        \State \( u_{n+1} = \frac{1}{k_{\text{max}}} \sum_{j=1}^{k_{\text{max}}} u_{n}^{j} \)
        \EndFor
    \end{algorithmic}
\end{algorithm}

Currently, the preconditioner $P$ is fixed for each iteration $n$, but later, we
will show how it can be chosen non-stationary varying with $n$.

To conclude this first part regarding the PNPD method, we prove a convergence result toward the solution of the initial problem \eqref{eq:model_PNPD} under suitable assumptions.

\begin{theorem}[Convergence of PNPD]
    \label{thm:convergence_pnpd}
    Let $f(u)=\frac{1}{2} \|Au-b^{\delta}\|^2$. Suppose that $h$ and $W$ satisfy Assumption \ref{hyp: standard}. Let $\{(u_n, v_n^0)\}_{n\in \mathbb{N}}$ be the primal--dual sequence generated by the PNPD method (Algorithm \ref{alg:pnpd}) with $\alpha_n = \alpha \in \left(0, \| P^{-1} A^T A \|^{-1}\right]$ and $\beta_n = \beta \in (0, \|W\|^{-2})$ for all $n\in \mathbb{N}$.
    Suppose also that the inertial parameters $\{\gamma_n\}_{n\in\mathbb{N}}$
    satisfy the condition \eqref{eq:conv_NPD_gamma} and that $S\in\mathcal{S}_+(\R^s)$
    and $P\in\mathcal{S}_+(\R^d)$ satisfy equation \eqref{eq:commutative_PNPD}.

    Then, the following statements hold:
    \begin{itemize}
        \item[(i)] the sequence $\{(u_n,v_n^0)\}_{n\in\mathbb{N}}$ is bounded;
        \item[(ii)] the primal sequence $\{u_n\}_{n\in \mathbb{N}}$ converges to a solution of the initial problem~\eqref{eq:model_PNPD}.
    \end{itemize}
\end{theorem}

\begin{proof}
    As anticipated, the NPD method applied to the problem \eqref{eq:model_PNPD} gives exactly the iteration of PNPD. We only have to prove that $\nabla f_S$ is $L_S$-Lipschitz continuous with $L_S = \| P^{-1} A^T A \|$. Thanks to equation \eqref{eq:GfS}, we have
    \begin{equation*}
        \| \nabla f_S(x) - \nabla f_S(y) \| =  \| P^{-1} \nabla f(x) - P^{-1} \nabla f(y) \|
        \leq \|P^{-1} A^T A \| \| x - y \|.
    \end{equation*}

    The thesis follows applying Theorem \ref{thm:convergence_npd} to problem \eqref{eq:model_PNPD}.
\end{proof}

\subsection{A polynomial choice for P}
A reasonable choice for the preconditioner $P$ and the associated matrix $S$, that satisfies the condition \eqref{eq:commutative_PNPD}, is
\begin{equation}
    \label{eq:pnpd_P_numerical}
    P = A^T A + \nu I, \qquad S = A A^T + \nu I,
\end{equation}
with $\nu > 0$. This is inspired by the iterated Tikhonov method, which coincides with the Levenberg–Marquardt method applied to linear problems \cite{Engl1996-yp}. More in general, we can show that the identity in equation \eqref{eq:commutative_PNPD} is satisfied whenever $P$ is in the set of some particular polynomials of $A^TA$ and $S$ is a corresponding polynomial of $AA^T$.

Defined the set of polynomials with non-negative coefficients as
\begin{equation}
    \mathcal{Q} = \left\{ \sum_{i=0}^{k} c_i x^i \mid c_i\ge0,\ c_0>0,\  k\in\mathbb{N} \right\},
\end{equation}
and the set of matrices
\begin{equation}
    \mathcal{P} = \left\{P\in\R^{d\times d}\ |\ P=p(A^T A)\ \wedge\ p\in\mathcal{Q} \right\},
\end{equation}
if $P \in \mathcal{P}$, then, by definition $P=p(A^TA)$ for some $p\in \mathcal{Q}$, and thus $P\in \mathcal{S}_+(\mathbb{R}^d)$ since $c_0 >0$.
We also observe that for each $P = p(A^TA) \in \mathcal{P}$, we can define a corresponding $S = p(A A^T) \in \mathcal{S}_+(\mathbb{R}^s)$,
which satisfies the equation~\eqref{eq:commutative_PNPD}.
Then, the following results hold.
\begin{proposition}
    \label{pnpd_P_commutation_rule_polynomials}
    Let $A\in\mathbb{R}^{s\times d}$, $P = p(A^T A)\in\mathcal{P}$ and $S = p(A A^T)$. Then it holds
    \begin{equation}\label{eq:PSthesis}
        P^{-1}A^T = A^T S^{-1}.
    \end{equation}
\end{proposition}
\begin{proof}
    Since $P\in\mathcal{S}_+(\mathbb{R}^d)$ and $S\in\mathcal{S}_+(\mathbb{R}^s)$, they are invertible. Moreover, it holds
    \begin{equation*}
        A^T S = A^T p(A A^T)                     = \sum_{i=0}^{k} c_i A^T (A A^T)^i  = \sum_{i=0}^{k} c_i (A^T A)^i A^T  = p(A^T A) A^T                     = P A^T,
    \end{equation*}
    which is equivalent to \eqref{eq:PSthesis}.
\end{proof}

\begin{proposition}
    \label{prop:norm_P_inv_A^T_A}
    Let $p(x) = \sum_{i=0}^{k} c_i x^i \in\mathcal{Q}$, $A\in\mathbb{R}^{s\times d}$, and $P = p(A^T A)\in \mathcal{P}$.
    Then, it holds
    \begin{equation}
        \|P^{-1} A^T A\| \le \left(c_1 + c_0 \|A\|^{-2}\right)^{-1}.
    \end{equation}
\end{proposition}

\begin{proof}
    Using the singular value decomposition of $A = U \Sigma V^T$, we have
    $A^TA=V \Sigma^T \Sigma V^T$, and hence
    \begin{equation*}
        P^{-1} A^T A = p(A^T A)^{-1} A^T A= p(V \Sigma^T \Sigma V^T)^{-1} V \Sigma^T \Sigma V^T  = V D V^T,
    \end{equation*}
    where $D = p(\Sigma^T \Sigma)^{-1} \Sigma^T \Sigma \in \R^{d \times d}$ is a diagonal matrix.
    Let $r$ be the rank of $A$ and  $\sigma_1 \geq \sigma_2 \geq \dots \geq \sigma_r>0$ its positive singular values, then the diagonal entries of $D$ are
    \begin{equation}
        d_j=
        \begin{cases}
            \frac{\sigma_j^2}{p(\sigma_j^2)}, & j=1,\dots, r,    \\
            0,                                & j = r+1,\dots,d.
        \end{cases}
    \end{equation}
    Since $c_i\geq 0$ for $i=0,\dots,k$, we have
    \begin{equation}\label{eq:dim_polynomial}
        p(\sigma_j^2) \ge c_1 \sigma_j^2 + c_0.
    \end{equation}
    Dividing both sides of equation \eqref{eq:dim_polynomial} by $\sigma_j^2$, we have
    \begin{equation*}
        \frac{p(\sigma_j^2)}{\sigma_j^2} \ge c_1 + \frac{c_0}{\sigma_j^2} \ge c_1 + \frac{c_0}{\sigma_1^2} = c_1 + c_0 \|A\|^{-2},
    \end{equation*}
    which is equivalent to
    \begin{equation}\label{eq:magci}
        \frac{\sigma_j^2}{p(\sigma_j^2)} \le \left(c_1 + c_0 \|A\|^{-2}\right)^{-1}.
    \end{equation}
    The thesis follows by observing that
    \[ \|P^{-1} A^T A\| = \|D\| =
        \max_{j=1,\dots,r} \frac{\sigma_j^2}{p(\sigma_j^2)}
        \leq \left(c_1 + c_0 \|A\|^{-2}\right)^{-1},
    \]
    thanks to equation \eqref{eq:magci}.
\end{proof}

From Proposition \ref{prop:norm_P_inv_A^T_A} and Theorem \ref{thm:convergence_pnpd} follows that, if $P\in\mathcal{P}$, the convergence of PNPD is guaranteed whenever we choose $0<\alpha<\left(c_1 + c_0 \|A\|^{-2}\right)$.

\subsection{PNPD as an inexact version of FISTA and ITTA}
When the proximity operator
$\prox_{\alpha_k h \circ W}$
in equation \eqref{eq:PNPD_scheme} is known in closed form, then the PNPD method \eqref{eq:PNPD_scheme} can be rewritten replacing the ``$\approx$'' with the ``$=$'' symbol.
This is the case explored in this subsection, where we consider
\begin{equation}\label{eq:h1}
    h(Wu) =\lambda \|u\|_1,
\end{equation} such that
\begin{equation}
    \label{pnpd_prox_soft}
    \prox_{\alpha_n h \circ W} (u) = \prox_{\alpha_n \lambda \|\cdot\|_1} (u) = \bm{S}_{\alpha_n \lambda}(u),
\end{equation}
where $\bm{S}_{\alpha}$ is the soft-thresholding function
\begin{equation*}
    \bm{S}_{\alpha}(u) = {\rm sign}(u)(|u|-\alpha)_+.
\end{equation*}
Note that in the regularization term \eqref{eq:h1}, usually $u$ represents the wavelet coefficients of the image, and the matrix $A$ in equation \eqref{eq:problem_leastsquares} is given by the product between the convolution operator and the inverse wavelet transform.

Setting $P=I$, equation \eqref{eq:PNPD_scheme} becomes
\begin{align*}
    \begin{cases}
        \bar u_n = u_n + \gamma_n (u_n - u_{n-1}),                                   \\
        u_{n+1} = \bm{S}_{\alpha_n \lambda}(\bar u_n - \alpha_n \nabla f(\bar u_n)), \\
    \end{cases}
\end{align*}
that is the PNPD iterations coincide with the iterative scheme of the
\textit{Fast Iterative Shrinkage--thresholding Algorithm} (FISTA) \cite{Beck-Teboulle-2009b}
applied to the optimization problem
\begin{equation}
    \label{eq:opt_problem_FISTA_as_exact_PNPD}
    \argmin_{u\in \R^d} f(u) + \lambda \|u\|_1.
\end{equation}
Of course, removing the extrapolation step, i.e., by setting $\gamma_n=0$, we have the iterations of the \textit{Iterative Shrinkage Thresholding Algorithm} (ISTA) \cite{Daubechies-et-al-2004}.

Consider now the variational model \eqref{eq:problem_leastsquares} and
fix $P = A^T A + \nu I$, we obtain
\begin{equation}
    \label{pnpd_itta_gradient}
    \begin{split}
        P^{-1}\nabla f(u) = P^{-1}A^T(Au-b^{\delta}) = A^T(AA^T+\nu I)^{-1}(Au-b^{\delta}).
    \end{split}
\end{equation}
Therefore, choosing $\alpha_n=1$ and $\gamma_n=0$ in \eqref{eq:PNPD_scheme}, the exact version of PNPD without extrapolation becomes
\begin{equation*}
    u_{n+1} = \bm{S}_{\lambda}\left(u_n - A^T(AA^T+\nu I)^{-1}(Au_n-b^{\delta})\right),
\end{equation*}
which is the
\textit{Iterated Tikhonov Thresholding Algorithm} (ITTA) proposed in~\cite{Huang-etal-2013} for solving
the optimization problem
\begin{equation*}
    \argmin_{u\in\R^d} \frac{1}{2}\|Au-b^{\delta}\|^2_{(AA^T+\nu I)^{-1}} + \lambda \|u\|_1 = \argmin_{u\in\R^d} \frac{1}{2}\|Au-b^{\delta}\|^2_{S^{-1}} + \lambda \|u\|_1,
\end{equation*}
with $S=AA^T+\nu I$. Note the similarity with our variational problem \eqref{eq:model_PNPD}.

\subsection{PNPD with a non-stationary preconditioner}
\label{subsec:PNPD_non_stationary}

The preconditioner $P = A^T A + \nu I$ used in \cite{NPDIT,Huang-etal-2013}, requires the estimation of the parameter $\nu$ that affects the convergence speed and stability of the method, c.f. subsection \ref{ssec:stab}.
Therefore, as done for non-stationary iterated Tikhonov, $\nu$ could be chosen as a non-stationary sequence as in \cite{NPDIT,Huang-etal-2013,cai2016regularization,donatelli2012nondecreasing}. In a more general framework, $P_n$ could change at each iteration and the PNPD method~\eqref{eq:PNPD_scheme} becomes
\begin{equation}
    \label{eq:PNPD_nonstat}
    \begin{cases}
        \bar{u}_n= u_n + \gamma_n(u_n-u_{n-1}), \\
        u_{n+1} \approx \prox_{\alpha_n h \circ W}(\bar u_n - \alpha_n P_n^{-1} \nabla f(\bar u_n)).
    \end{cases}
\end{equation}

In this paper, we propose a class of preconditioners $P_n$ of the form
\begin{equation}
    \label{eq:non_stationary_preconditioner}
    P_n = (1-\nu_n)A^TA + \nu_n I,
\end{equation}
with $\{\nu_n\}_{n\in\mathbb{N}}$ such that $0<\nu_n\leq 1$ for all $n\in\mathbb{N}$. We propose three possible choices for the sequence $\{\nu_n\}_{n\in\mathbb{N}}$ that are
\begin{equation}
    \begin{split}
         & \nu_n = \frac{0.85^n}{2} + \nu_\infty \quad \text{with} \quad \nu_\infty \in \left(0,\frac{1}{2}\right); \label{eq:litterature_scheduler}
    \end{split}
\end{equation}
\begin{equation}
    \begin{split}
         & \nu_n = \left(1-\frac{1}{\sqrt{n+1}}\right)(1-\nu_0) + \nu_0 \quad \text{with} \quad \nu_0 \in (0,1); \label{eq:npdit_scheduler}
    \end{split}
\end{equation}
\begin{equation}
    \begin{split}
         & \nu_n = \min\{c^{n-n_{\text{bt}}}, 1\} \quad \text{with} \quad n_\text{bt}\in \mathbb{N}, \quad c = \nu_0^{-\frac{1}{n_\text{bt}}},\quad \text{and} \quad \nu_0 \in (0,1).
        \label{eq:bootstrap_scheduler}
    \end{split}
\end{equation}

The decreasing sequence \eqref{eq:litterature_scheduler} is largely used in the literature \cite{Huang-etal-2013,cai2016regularization,NPDIT} and is inspired by the seminal paper on the convergence of the non-stationary iterated Tikhonov method \cite{hanke1998nonstationary}. On the other hand, nondecreasing sequences, if properly chosen, can provide fast and stable convergence as proved in \cite{donatelli2012nondecreasing}. This is the case of the sequence \eqref{eq:npdit_scheduler} proposed in \cite{NPDIT} for NPDIT.
The sequence \eqref{eq:bootstrap_scheduler} is also an increasing
sequence but with a different growth rate than the sequence~\eqref{eq:npdit_scheduler}.
Indeed, the latter sequence starts from $\nu_0$ and increases exponentially until it reaches $1$ for $n\geq n_\text{bt}$.

Note that, for $n$ large enough, $\nu_n$ chosen as in equation \eqref{eq:litterature_scheduler} approaches $\nu_\infty$, while $\nu_n \to 1$ when chosen according to  \eqref{eq:npdit_scheduler} or \eqref{eq:bootstrap_scheduler}.
Therefore, the convergence of the non-stationary PNPD iteration \eqref{eq:PNPD_nonstat} follows from Theorem \ref{thm:convergence_pnpd} for the $\nu_n$ sequence \eqref{eq:litterature_scheduler}, while follows from Theorem~\ref{thm:convergence_npd} for the $\nu_n$ sequences in
\eqref{eq:npdit_scheduler} and \eqref{eq:bootstrap_scheduler}.
In particular, for the latter sequence \eqref{eq:bootstrap_scheduler} we have
$P_n = I$ for each $n\geq n_\text{bt}$; consequently, the method becomes equivalent to NPD when $n\geq n_\text{bt}$.
In conclusion, when
$\nu_n$ is defined as in equation \eqref{eq:litterature_scheduler}, the non-stationary PNPD method \eqref{eq:PNPD_nonstat}
computes the solution of the problem \eqref{eq:model_PNPD} with the matrix $S$ depending on $\nu_\infty$.
While, when
the $\nu_n$ sequence is chosen as in
\eqref{eq:npdit_scheduler} or \eqref{eq:bootstrap_scheduler},
it computes the solution of the model \eqref{eq:problem_leastsquares}
addressing the problem of replacing the $\ell^2$-norm in the data fidelity term with the norm induced by~$S$.

When we choose a non-stationarity preconditioner $P_n$, each iteration of the non-stationary PNPD method \eqref{eq:PNPD_nonstat} can be seen as a single step to solve the optimization problem
\begin{equation}\label{eq:modelSn}
    \argmin_{u\in\mathbb{R}^d} \frac{1}{2}\|Au-b^\delta\|^2_{S_n^{-1}} + \lambda\ h(Wu).
\end{equation}
Since the data fidelity term is a function of $S_n^{-1}$ while the regularization term does not depend on $n$, there is an implicit dependence on $n$ of the regularization parameter~$\lambda$. This can be noted, normalizing $S_n^{-1}$ in \eqref{eq:modelSn}, i.e., replacing it with $\frac{S_n^{-1}}{\|S_n^{-1}\|}$. With this change, the problem \eqref{eq:modelSn} becomes
\begin{equation*}
    \argmin_{u\in\mathbb{R}^d} \frac{1}{2}\|Au-b^\delta\|^2_{\frac{S_n^{-1}}{\|S_n^{-1}\|}} + \lambda\ h(Wu)\\
    =
    \argmin_{u\in\mathbb{R}^d} \frac{1}{2}\|Au-b^\delta\|^2_{S_n^{-1}} + \lambda_n\ h(Wu),
\end{equation*}
with $\lambda_n=\lambda\|S_n^{-1}\|$.
Again, for $n$ large enough, the convergence of $S_n$ follows the analysis above for $P_n$, and in particular, $S_n$ converges to the identity matrix for the $\nu_n$ sequences
\eqref{eq:npdit_scheduler} and \eqref{eq:bootstrap_scheduler}.

\section{Numerical results}
\label{sec:numerical_results}
This section presents numerical results for the PNPD method (Algorithm~\ref{alg:pnpd}). Specifically, we compare the results obtained with our proposed algorithm against those achieved with NPD (Algorithm \ref{alg:npd}) and NPDIT (Algorithm \ref{alg:npdit}).
An extensive comparison of NPDIT with accelerated versions of proximal gradient methods in the literature can be found in \cite{NPDIT}.

All the algorithms are implemented in Python 3.12.3 with NumPy 1.26.4 and all the code used to generate the plots in this section is available in a public Git repository\footnote{\url{https://github.com/Giuseppe499/PNPD}}. The experiments below were run on a PC with Kubuntu 24.04, equipped with a 4.75 GHz AMD Ryzen 7 7735HS processor and 16 GB of RAM.

To evaluate the performance of the methods, we use two metrics. The first is the \textit{Relative Reconstruction Error} (RRE), defined as
\begin{equation}
    \label{eq:RRE}
    \text{RRE}(x_{\text{rec}}) = \frac{\|x_{\text{true}} - x_{\text{rec}}\|}{\|x_{\text{true}}\|},
\end{equation}
where \(x_{\text{rec}}\) is the reconstructed image and \(x_{\text{true}}\) is the ground truth. The second metric is the \textit{Structural Similarity Index} (SSIM), which is computed using the \texttt{skimage.metrics.structural\_similarity} function from the \textit{scikit-image} library.

In this section, we often depict those performance metrics as a function of elapsed time. This is computed using the Python function \texttt{time.perf\_counter}. We do not count the initialization time, since it is comparable for all the considered algorithms and is negligible with respect to a single iteration of NPD.

Unless stated otherwise, we use the following parameter choices to ensure the convergence of the methods:

\begin{itemize}
    \item For NPD, PNPD, and NPDIT, we set \(\alpha_n = \alpha = 1\), and \(u_0 = b^{\delta}\).
    \item For NPD and PNPD, we set \(\beta_n = \beta = \frac{0.99}{8}\).
    \item For NPDIT, we set \(\beta_n = \beta = \frac{0.99\ \nu}{8}\).
\end{itemize}

When comparing the two preconditioning strategies, NPDIT and PNPD, we use a stationary version of the preconditioner, \(P = A^T A + \nu I\), where \(\nu > 0\) will be specified for each example. To provide a comprehensive analysis of the performance of our left preconditioning strategy, we also consider the non-stationary preconditioner \(P_n = (1-\nu_n)A^T A + \nu_n I\), where the sequence \(\{\nu_n\}_{n\in\mathbb{N}}\) is chosen as in equations \eqref{eq:litterature_scheduler}--\eqref{eq:bootstrap_scheduler}.

\def \ex {1}
\def \preposIm{a }
\def \imName {cameramen}
\def \psfComment{a Gaussian PSF with standard deviation $\sigma=2$ pixels}
\def \noise{0.01}
\def \noisePercent{1}
\def \lam {2\cdot10^{-4}}
\def \lamPNPD {2\cdot10^{-3}}

\subsection{Example \ex}
\label{sec:example\ex}

In this example, we considered a grayscale image of \preposIm\imName\ with dimensions $256 \times 256$. The blurred image $b$ was obtained using a Gaussian PSF with a $\sigma = 2$ pixels standard deviation. To generate the final observed image~$b^{\delta}$, we added white Gaussian noise $\eta_\delta$  generated using the NumPy function \texttt{numpy.random.normal}. $\eta_\delta$ is scaled to satisfy $\delta = \|\eta_\delta\| = \noise\|b\|$.
Figure \ref{fig:example_\ex problem} displays the ground truth image $x$, the PSF (shifted
to the center and cropped for better visualization), and the observed image $b^{\delta}$. In Figure
\ref{fig:example_\ex PNPD_comparison_reconstructions} we show
the reconstructions obtained with NPD, NPDIT, and PNPD after 10 iterations.

\def \folder{/example_\ex /problem_}
\begin{figure}
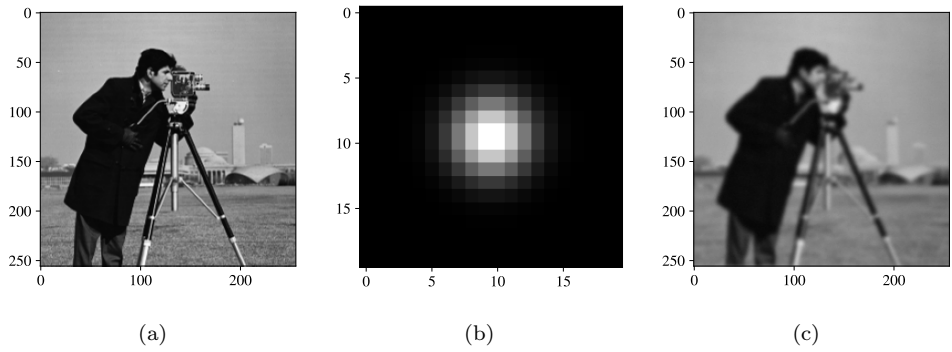

    \centering
    \makebox[\textwidth][c]{
        \begin{subfigure}{\tPlotWidth}
            \centering
            \includegraphics[width=\textwidth]{\folder ground_truth.pdf}
            \caption{\label{fig:example_\ex ground_truth}}
        \end{subfigure}
        \begin{subfigure}{\tPlotWidth}
            \centering
            \includegraphics[width=\textwidth]{\folder psf.pdf}
            \caption{\label{fig:example_\ex psf}}
        \end{subfigure}
        \begin{subfigure}{\tPlotWidth}
            \centering
            \includegraphics[width=\textwidth]{\folder observed.pdf}
            \caption{\label{fig:example_\ex observed}}
        \end{subfigure}
    }
    \caption{Example \ex: (\subref{fig:example_\ex ground_truth})
        Ground truth image of \preposIm\imName.
        (\subref{fig:example_\ex psf}) PSF used to blur the ground truth
        (center crop of size $20\times20$).
        (\subref{fig:example_\ex observed}) Observed image $b^\delta$
        obtained by adding white Gaussian noise on top of the
        discrete circular convolution of the ground truth and the PSF. The Gaussian noise $\eta_{\delta}$ is such that $\|\eta_{\delta}\| = \noise\|b^\delta\|$.}
    \label{fig:example_\ex problem}
\end{figure}

\def \folder{/example_\ex /PNPD_comparison/}
\def \filePrefix{reconstruction_it=10_}
\begin{figure}
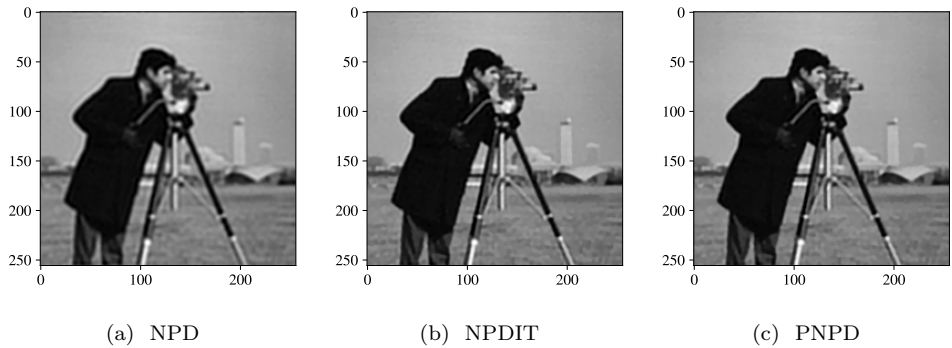

    \centering
    \makebox[\textwidth][c]{
        \begin{subfigure}{\tPlotWidth}
            \centering
            \includegraphics[width=\textwidth]{\folder\filePrefix NPD.pdf}
            \caption{
                \label{fig:example_\ex reconstruction NPD} NPD}
        \end{subfigure}
        \begin{subfigure}{\tPlotWidth}
            \centering
            \includegraphics[width=\textwidth]{\folder\filePrefix NPDIT.pdf}
            \caption{\label{fig:example_\ex reconstruction NPDIT} NPDIT}
        \end{subfigure}
        \begin{subfigure}{\tPlotWidth}
            \centering
            \includegraphics[width=\textwidth]{\folder\filePrefix PNPD.pdf}
            \caption{\label{fig:example_\ex reconstruction PNPD} PNPD}
        \end{subfigure}
    }
    \caption{Example \ex: Comparison of the reconstructions obtained with
    NPD, PNPD, and NPDIT after 10 iterations.
    The preconditioner parameter is $\nu=10^{-1}$.
    The number of nested loop iterations is $k_{\text{max}}=3$.
    The regularization parameter is $\lambda=2\cdot10^{-4}$ for NPD and
    NPDIT, while $\lambda=2\cdot10^{-3}$ for PNPD.}
    \label{fig:example_\ex PNPD_comparison_reconstructions}
\end{figure}

The comparison between the three different algorithms NPD, NPDIT, and PNPD is presented in Figure \ref{fig:example_\ex PNPD_comparison}. The performances of each method were measured through the RRE and SSIM functionals. Iteration-wise, we observe that PNPD and NPDIT exhibit similar
behaviors, both converging faster than NPD. This is due to the presence of the preconditioner $P$, which effectively enhances the speed of convergence of the algorithms. In terms of execution time, the PNPD strategy outperforms both the other methods. The gap between our proposal and NPDIT is due to the fact that as $k_{\mathrm{max}}$ increases, the NPDIT algorithm must compute more FFTs at each iteration. Indeed, the NPDIT method looks for approximate evaluations of $\prox^{P}_{\alpha h \circ W}$ while PNPD approximates $\prox_{\alpha h \circ W}$ in
the same manner as NPD. Therefore, NPDIT has to perform an extra
multiplication by $P^{-1}$ for each extra nested iteration compared to PNPD. This behavior is underlined in Table \ref{tab:example_\ex PNPD_NPDIT_k}, reporting the average time spent for one step of PNPD and NPDIT for different values of $k_{\text{max}}$. We can see
that $\Delta$, which is the difference between the average time spent for the two methods, increases as $k_{\text{max}}$ increases and,
more interestingly, the same happens for the ratio of the execution time of one
step of NPDIT and one of PNPD. The execution speed gap between the two methods increases to the point in which one iteration of PNPD with $k_{\text{max}}=8$ is faster than NPDIT with $k_{\text{max}}=1$.
\begin{figure}
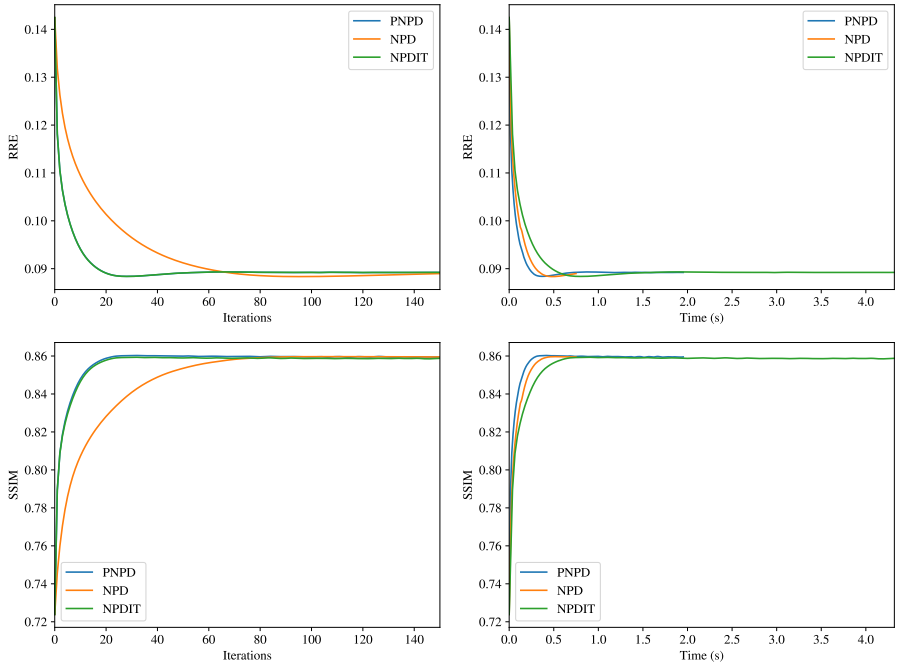

    \centering
    \makebox[\textwidth][c]{
        \begin{subfigure}{\qPlotWidth}
            \centering
            \includegraphics[width=\textwidth]{\folder RRE_iterations.pdf}
        \end{subfigure}
        \begin{subfigure}{\qPlotWidth}
            \centering
            \includegraphics[width=\textwidth]{\folder RRE_time.pdf}
        \end{subfigure}
    }
    \makebox[\textwidth][c]{
        \begin{subfigure}{\qPlotWidth}
            \centering
            \includegraphics[width=\textwidth]{\folder SSIM_iterations.pdf}
        \end{subfigure}
        \begin{subfigure}{\qPlotWidth}
            \centering
            \includegraphics[width=\textwidth]{\folder SSIM_time.pdf}
        \end{subfigure}
    }
    \caption{Example \ex: Comparison of the RREs and SSIMs between PNPD, NPD,
    and NPDIT. The preconditioner parameter is $\nu=10^{-1}$.
    The number of nested loop iterations is $k_{\text{max}}=1$ for NPD
    and $k_{\text{max}}=3$ for NPDIT and PNPD.
    The regularization parameter is $\lambda=2\cdot10^{-4}$ for NPD and
    NPDIT, while $\lambda=2\cdot10^{-3}$ for PNPD.}
    \label{fig:example_\ex PNPD_comparison}
\end{figure}

\begin{table}
    \centering
    \begin{tabular}{|c|c|c|c|c|}
\hline
   $k_{{max}}$ &      PNPD &     NPDIT &   $\Delta$ &   NPDIT/PNPD \\
\hline
             1 & 0.0107 & 0.0196 &  0.0088 &      1.822 \\
             2 & 0.0109 & 0.0245 &  0.0135 &      2.239 \\
             4 & 0.0125 & 0.0372 &  0.0247 &      2.971  \\
             8 & 0.0194 & 0.0617 &  0.0423 &      3.177 \\
            16 & 0.0239 & 0.0933 &  0.0694 &      3.894 \\
            32 & 0.0357 & 0.1508  &  0.1151  &      4.216 \\
            64 & 0.0584 & 0.2703  &  0.2119  &      4.629 \\
\hline
\end{tabular}
    \caption{Average time spent for one step of PNPD and NPDIT for different values of $k_{\text{max}}$. The difference between the execution time of the two methods is shown in the $\Delta$ column, while the last column reports the ratio of the execution time of a step of NPDIT and one of PNPD.}
    \label{tab:example_\ex PNPD_NPDIT_k}
\end{table}

\subsubsection{Stability}\label{ssec:stab}
While the previous analysis was dedicated to comparing PNPD with NPD, and NPDIT, we are now considering the stability properties of our proposal. In
Figure \ref{fig:example_\ex PNPD_nu_PNPD_k=1} we report the RREs and SSIMs obtained for PNPD with different values of the preconditioner parameter $\nu$ in equation (\ref{eq:pnpd_P_numerical}).  We can observe that as $\nu$ decreases the method becomes faster but also unstable. This behavior is justified from a theoretical viewpoint. Indeed, as $\nu$ approaches zero, the preconditioner $P = A^TA+\nu I$ tends to the positive semidefinite linear operator $A^TA$. Therefore, computing $P^{-1}$ we are trying to invert an almost singular operator and this introduces instability along the iterations. On the other hand, the preconditioner $P$ is converging to the true Hessian matrix of the differentiable term in the model problem \eqref{eq:problem_leastsquares}, resulting in faster convergence towards the minimum point. A trade-off between these two properties is clearly needed.

One possible strategy to overcome instability is to increase the number $k_{\text{max}}$ of nested iterations. In Figure \ref{fig:example_\ex PNPD_nu_PNPD} we show the
results obtained with PNPD with different values of $\nu$, but in this case, for
each different choice, we increased $k_{\text{max}}$ enough to prevent
instability. Even though a higher $k_{\text{max}}$ reflects in a higher iteration
execution time, a decrease in $\nu$ still results in a faster convergence.
\def \folder{/example_\ex /PNPD_nu/PNPD_k=1/}
\begin{figure}
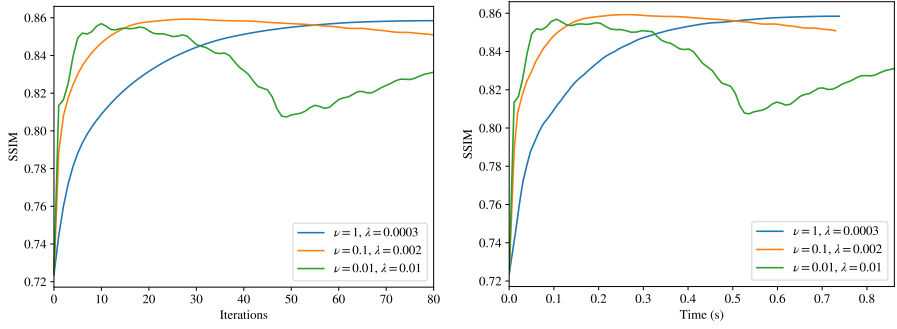

    \centering
    \makebox[\textwidth][c]{
        \begin{subfigure}{\qPlotWidth}
            \centering
            \includegraphics[width=\textwidth]{\folder SSIM_iterations.pdf}
        \end{subfigure}
        \begin{subfigure}{\qPlotWidth}
            \centering
            \includegraphics[width=\textwidth]{\folder SSIM_time.pdf}
        \end{subfigure}
    }
    \caption{Example \ex: Comparison of the SSIMs of PNPD with $k_{\text{max}}=1$ for
        different values of $\nu$ and $\lambda$.}
    \label{fig:example_\ex PNPD_nu_PNPD_k=1}
\end{figure}

\def \folder{/example_\ex /PNPD_nu/PNPD/}
\begin{figure}
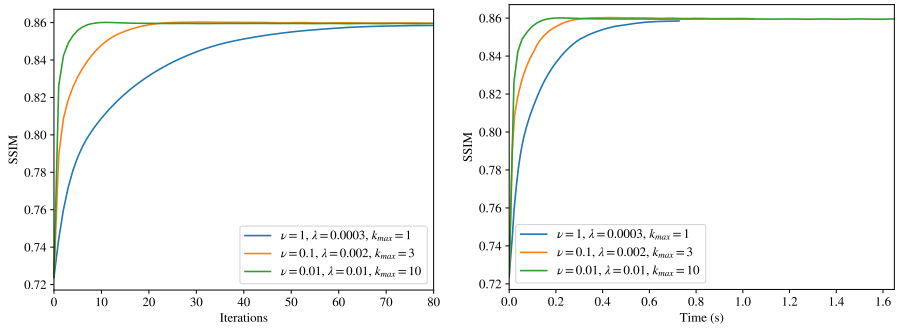

    \centering
    \makebox[\textwidth][c]{
        \begin{subfigure}{\qPlotWidth}
            \centering
            \includegraphics[width=\textwidth]{\folder SSIM_iterations.pdf}
        \end{subfigure}
        \begin{subfigure}{\qPlotWidth}
            \centering
            \includegraphics[width=\textwidth]{\folder SSIM_time.pdf}
        \end{subfigure}
    }
    \caption{Example \ex: Comparison of the SSIMs of PNPD for different
        values of $\nu$, $\lambda$, and $k_{\text{max}}$ set high enough
        to fix instability.}
    \label{fig:example_\ex PNPD_nu_PNPD}
\end{figure}

One other possible strategy to overcome instability along the iterations is to get rid of the extrapolation step obtaining the PNPD\_NE method defined~as
\begin{equation}
    \label{eq:PNPD_NE}
    u_{n+1} \approx \prox_{\alpha_n h \circ W}(u_n - \alpha_n P^{-1} \nabla f(u_n)).
\end{equation}
This method is less prone to instability than PNPD as depicted in Figure \ref{fig:example_\ex PNPD_nu_PNPD_NE_k=1}. Indeed,
differently from Figure \ref{fig:example_\ex PNPD_nu_PNPD_k=1}, even though $k_{\text{max}}=1$, PNPD\_NE does not show worrying instability even for $\nu=10^{-2}$. However, this improvement in stability comes at the cost of convergence speed.

Therefore, in the case in which PNPD shows
instability for a certain value of $\nu$, we have two different choices. We can increase $k_{\text{max}}$
(which also comes at the cost of CPU time)
or we can use PNPD\_NE instead of PNPD.
As an example, analyzing the case in which $\nu=0.01$, we can see from Figure \ref{fig:example_\ex PNPD_nu_PNPD} that PNPD with $k_{\text{max}}=10$ reaches an SSIM of $0.853$ after $0.079$ seconds or $4$ iterations, while PNPD\_NE (see Figure \ref{fig:example_\ex PNPD_nu_PNPD_NE_k=1}) after $0.093$ seconds or $8$ iterations.

\def \folder{/example_\ex /PNPD_nu/PNPD_NE_k=1/}
\begin{figure}
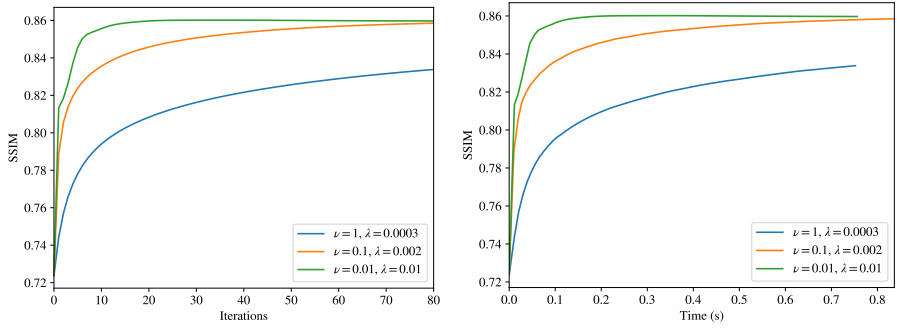

    \centering
    \makebox[\textwidth][c]{
        \begin{subfigure}{\qPlotWidth}
            \centering
            \includegraphics[width=\textwidth]{\folder SSIM_iterations.pdf}
        \end{subfigure}
        \begin{subfigure}{\qPlotWidth}
            \centering
            \includegraphics[width=\textwidth]{\folder SSIM_time.pdf}
        \end{subfigure}
    }
    \caption{Example \ex: Comparison of the SSIMs of PNPD\_NE
        with $k_{\text{max}}=1$ for different values of $\nu$ and $\lambda$.}
    \label{fig:example_\ex PNPD_nu_PNPD_NE_k=1}
\end{figure}

\subsubsection{Non--stationary Preconditioning}
Lastly, we compare the results obtained for PNPD with the
preconditioner sequences $\{P_n\}$ of the form in equation
\eqref{eq:non_stationary_preconditioner} with the different
sequences of $\{\nu_n\}_{n_\in\N}$ discussed in Subsection
\ref{subsec:PNPD_non_stationary}. For the choice
\eqref{eq:litterature_scheduler} we set $\nu_\infty=10^{-2}$,
for \eqref{eq:npdit_scheduler} and
\eqref{eq:bootstrap_scheduler} we set $\nu_0=10^{-2}$, and for \eqref{eq:bootstrap_scheduler} we set $n_\text{bt}=20$.
The compared stationary PNPD, cf.
\eqref{eq:litterature_scheduler}--\eqref{eq:bootstrap_scheduler},
is obtained with $\nu_n=\nu=10^{-1}$.

In Figure \ref{fig:example_\ex PNPD_non_stationary},
we observe that the decreasing sequence
\eqref{eq:litterature_scheduler} is slower at the beginning
and then accelerates, overtaking its stationary version.
On the other hand, the increasing sequence
\eqref{eq:npdit_scheduler},
has a fast convergence in the first iterations, but it slows
down, and it falls behind the stationary PNPD. The bootstrap
sequence \eqref{eq:bootstrap_scheduler} also exhibits an increasing behavior, resulting in fast convergence in the initial steps where the preconditioner acceleration is most noticeable. Furthermore, differently from \eqref{eq:npdit_scheduler}, we can see that \eqref{eq:bootstrap_scheduler} can outperform the stationary PNPD.

In Figure \ref{fig:example_\ex PNPD_bootstrap}, we tested the behavior of the bootstrap sequence \eqref{eq:bootstrap_scheduler} when the parameters $\nu_0$ and $n_\text{bt}$ are changed and are set unfavorably. We can see that by choosing $\nu_0$ too small $(\nu_0=10^{-3})$ the method becomes unstable in the first iterations, but, since it eventually becomes equivalent to NPD after a finite number of iterations, the method still converges to a solution of equation~\eqref{eq:problem_leastsquares}. If we instead choose a reasonable $\nu_0$, changing $n_\text{bt}$ can affect convergence speed but has a large margin of error. In particular, we tested $n_\text{bt}$ for the values, $5$, $20$ and $50$. We can see that the convergence speed increases as $n_\text{bt}$ increases. This is particularly noticeable when switching from $5$ to $20$, but is less impactful when changing from $20$ to $50$. Taking $n_\text{bt}$ too high results in the switch to NPD happening later and, therefore, if $\nu_0$ is chosen in a suboptimal way, the method could diverge at the beginning and will start to converge only when approaching $n_\text{bt}$.
\def \folder{/example_\ex /PNPD_non_stationary/}
\begin{figure}
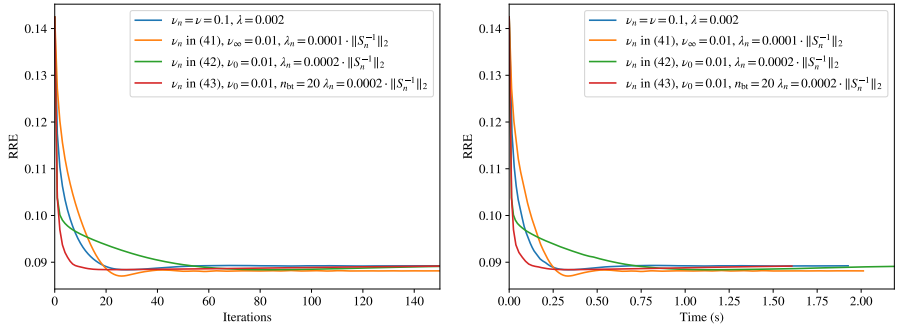

    \centering
    \makebox[\textwidth][c]{
        \begin{subfigure}{\qPlotWidth}
            \centering
            \includegraphics[width=\textwidth]{\folder RRE_iterations.pdf}
        \end{subfigure}
        \begin{subfigure}{\qPlotWidth}
            \centering
            \includegraphics[width=\textwidth]{\folder RRE_time.pdf}
        \end{subfigure}
    }
    \caption{Example \ex: Comparison of the RREs for the non-stationary version of PNPD and different sequences of $\nu_n$.
        For this test, we set $k_{\text{max}}=3$.}
    \label{fig:example_\ex PNPD_non_stationary}
\end{figure}

\def \folder{/example_\ex /PNPD_bootstrap/}
\begin{figure}
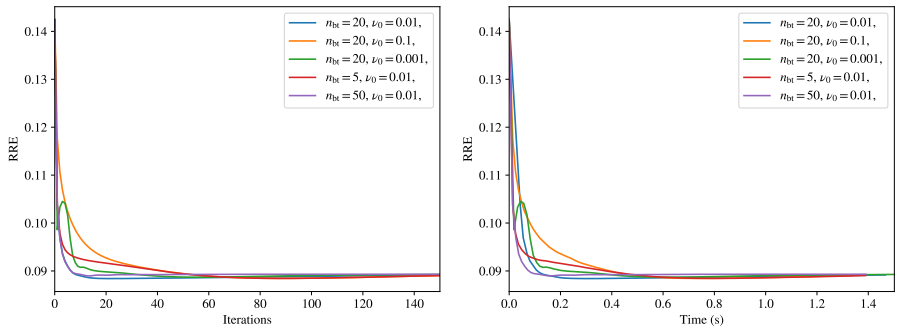

    \centering
    \makebox[\textwidth][c]{
        \begin{subfigure}{\qPlotWidth}
            \centering
            \includegraphics[width=\textwidth]{\folder RRE_iterations.pdf}
        \end{subfigure}
        \begin{subfigure}{\qPlotWidth}
            \centering
            \includegraphics[width=\textwidth]{\folder RRE_time.pdf}
        \end{subfigure}
    }
    \caption{Example \ex: Comparison of the RREs obtained with PNPD for the non-stationary bootstrap sequence of $\nu_n$ in \eqref{eq:bootstrap_scheduler}.
    For this test, we set $\lambda_n= 2\cdot10^{-4}\cdot \|S_n^{{-1}}\|$ and $k_{\text{max}}=3$.}
    \label{fig:example_\ex PNPD_bootstrap}
\end{figure}

\def \ex {2}
\def \preposIm{a }
\def \imName {cameramen}
\def \psfCommentNoProp{PSF with standard deviation $\sigma=2$ pixels}
\def \psfComment{a \psfCommentNoProp}
\def \noise{0.02}
\def \noisePercent{2}

\subsection{Example 2: cameraman with 2\% Gaussian noise}
\label{sec:example\ex}

In this example, we consider the same framework as in the previous one but increasing the noise level in the final observed image $b^{\delta}$ from $1\%$ to $2\%$.

To further analyze the performance of these algorithms, Figure \ref{fig:example_\ex PNPD_comparison} presents a comparison of the quality of the reconstructions at each step by computing the RRE and the SSIM. In both scenarios, the PNPD method outperforms the other strategies in terms of both the number of iterations required and CPU time.

\def \folder{/example_\ex /PNPD_comparison/}
\begin{figure}
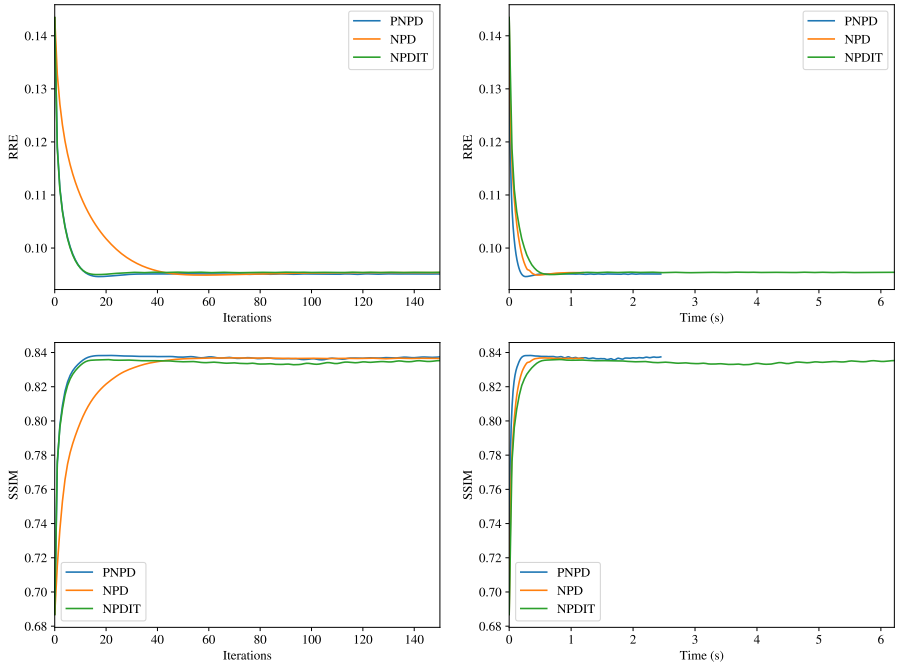

    \centering
    \makebox[\textwidth][c]{
        \begin{subfigure}{\qPlotWidth}
            \centering
            \includegraphics[width=\textwidth]{\folder RRE_iterations.pdf}
        \end{subfigure}
        \begin{subfigure}{\qPlotWidth}
            \centering
            \includegraphics[width=\textwidth]{\folder RRE_time.pdf}
        \end{subfigure}
    }
    \makebox[\textwidth][c]{
        \begin{subfigure}{\qPlotWidth}
            \centering
            \includegraphics[width=\textwidth]{\folder SSIM_iterations.pdf}
        \end{subfigure}
        \begin{subfigure}{\qPlotWidth}
            \centering
            \includegraphics[width=\textwidth]{\folder SSIM_time.pdf}
        \end{subfigure}
    }
    \caption{Example \ex: Comparison of the RREs and SSIMs between
    PNPD, NPD, and NPDIT.
    The preconditioner parameter is $\nu=10^{-1}$.
    The number of nested loop iterations is $k_{\text{max}}=2$
    for NPD and $k_{\text{max}}=5$ for NPDIT and PNPD.
    The regularization parameter is $\lambda=7\cdot10^{-4}$ for NPD and
    NPDIT, and is $\lambda=6\cdot10^{-3}$ for PNPD.}
    \label{fig:example_\ex PNPD_comparison}
\end{figure}

A further improvement of PNPD can be obtained by employing the
strategies discussed in the previous Example 1, as shown in
Figure \ref{fig:example_\ex PNPD_proposed}.
Here, in addition to the optimal stationary case presented in Figure
\ref{fig:example_\ex PNPD_comparison}, we consider a smaller $\nu$, increasing $k_{\text{max}}$ or using PNPD\_NE which is described in (\ref{eq:PNPD_NE}). In the same figure, we also consider the non-stationary PNPD with the bootstrap strategy, as defined in \eqref{eq:bootstrap_scheduler}.
As we already expected, the case without extrapolation is the slowest in terms of iterations. However, when there is uncertainty about the optimal parameters settings, the PNPD\_NE method appears to be a reasonable choice. This is because the method remains stable without the need to increase the number of nested iterations, and it allows for smaller values of $\nu$ in the preconditioner. Furthermore, in terms of CPU time, all the proposed strategies shows similar performances.

\def \folder{/example_\ex /PNPD_proposed/}
\begin{figure}
    \centering
    \makebox[\textwidth][c]{
        \begin{subfigure}{\qPlotWidth}
            \centering
            \includegraphics[width=\textwidth]{\folder RRE_iterations.pdf}
        \end{subfigure}
        \begin{subfigure}{\qPlotWidth}
            \centering
            \includegraphics[width=\textwidth]{\folder RRE_time.pdf}
        \end{subfigure}
    }
    \makebox[\textwidth][c]{
        \begin{subfigure}{\qPlotWidth}
            \centering
            \includegraphics[width=\textwidth]{\folder SSIM_iterations.pdf}
        \end{subfigure}
        \begin{subfigure}{\qPlotWidth}
            \centering
            \includegraphics[width=\textwidth]{\folder SSIM_time.pdf}
        \end{subfigure}
    }
    \caption{Example \ex: Comparison of the RREs and SSIMs between the proposed
    variants of PNPD. In particular, the results are obtained with PNPD,
    PNPD\_NE, and the bootstrap version of PNPD (PNPD\_BT),
    which uses the non-stationary sequence $\nu_n$ in
    \eqref{eq:bootstrap_scheduler} and $\lambda_n=\hat\lambda\cdot\|S_n^{{-1}}\|$.}
    \label{fig:example_\ex PNPD_proposed}
\end{figure}

\def \ex {3}
\def \preposIm{some }
\def \imName {peppers}
\def \psfComment{a motion blur PSF}
\def \noise{0.005}
\def \noisePercent{0.5}

\subsection{Example 3: peppers with 0.5\% Gaussian noise}
\label{sec:example\ex}
As a final example, we consider a different grayscale image of some peppers with dimensions $256 \times 256$, a motion blur PSF, and white Gaussian noise with an intensity level of $0.5\%$.

Figure \ref{fig:example_\ex problem} displays the ground truth $x$, the PSF, and the blurred image $b^{\delta}$. Figure \ref{fig:example_\ex PNPD_comparison_reconstructions} shows the reconstructions obtained with NPD, NPDIT, and PNPD after 5 iterations. We can observe that, even though the noise level is lower than in previous examples, the PNPD and NPDIT methods outperform the standard NPD strategy after just 5 iterations.
\def \folder{/example_\ex /problem_}
\begin{figure}
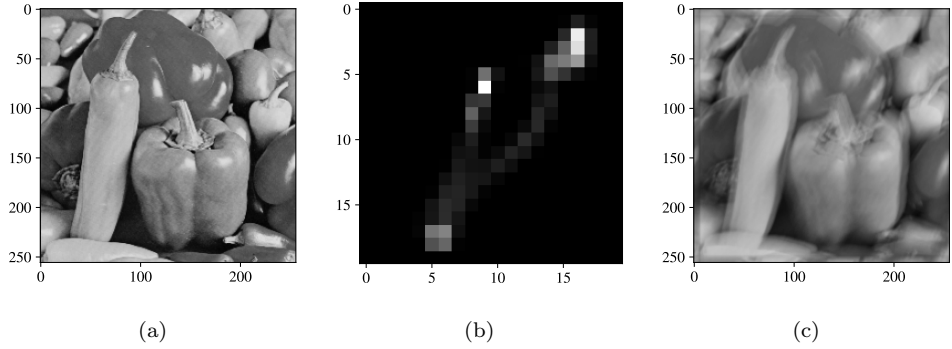

    \centering
    \makebox[\textwidth][c]{
        \begin{subfigure}{\tPlotWidth}
            \centering
            \includegraphics[width=\textwidth]{\folder ground_truth.pdf}
            \caption{\label{fig:example_\ex ground_truth}}
        \end{subfigure}
        \begin{subfigure}{\tPlotWidth}
            \centering
            \includegraphics[width=\textwidth]{\folder psf.pdf}
            \caption{\label{fig:example_\ex psf}}
        \end{subfigure}
        \begin{subfigure}{\tPlotWidth}
            \centering
            \includegraphics[width=\textwidth]{\folder observed.pdf}
            \caption{\label{fig:example_\ex observed}}
        \end{subfigure}
    }
    \caption{Example \ex: (\subref{fig:example_\ex ground_truth})
        Ground truth image of \preposIm\imName.
        (\subref{fig:example_\ex psf}) PSF used to blur the ground truth
        (center crop of size $20\times20$).
        (\subref{fig:example_\ex observed}) Observed image $b^\delta$
        obtained by adding white Gaussian noise on top of the
        discrete circular convolution of the ground truth and the PSF. The Gaussian noise $\eta_{\delta}$ is such that $\|\eta_{\delta}\| = \noise\|b^\delta\|$.}
    \label{fig:example_\ex problem}
\end{figure}

\def \folder{/example_\ex /PNPD_comparison/}
\def \filePrefix{reconstruction_it=5_}
\begin{figure}
    \centering
    \makebox[\textwidth][c]{
        \begin{subfigure}{\tPlotWidth}
            \centering
            \includegraphics[width=\textwidth]{\folder\filePrefix NPD.pdf}
            \caption{
                \label{fig:example_\ex reconstruction NPD} NPD}
        \end{subfigure}
        \begin{subfigure}{\tPlotWidth}
            \centering
            \includegraphics[width=\textwidth]{\folder\filePrefix NPDIT.pdf}
            \caption{\label{fig:example_\ex reconstruction NPDIT} NPDIT}
        \end{subfigure}
        \begin{subfigure}{\tPlotWidth}
            \centering
            \includegraphics[width=\textwidth]{\folder\filePrefix PNPD.pdf}
            \caption{\label{fig:example_\ex reconstruction PNPD} PNPD}
        \end{subfigure}
    }
    \caption{Example \ex: Comparison of the reconstructions obtained with
    NPD, PNPD, and NPDIT after 5 iterations.
    The preconditioner parameter is $\nu=10^{-2}$.
    The number of nested loop iterations is $k_{\text{max}}=2$.
    The regularization parameter is $\lambda=10^{-4}$ for NPD and
    NPDIT, and is $\lambda=6\cdot10^{-3}$ for PNPD.}
    \label{fig:example_\ex PNPD_comparison_reconstructions}
\end{figure}

This is particularly noticeable in Figure \ref{fig:example_\ex PNPD_comparison}, where we compared the RREs and SSIMs obtained with PNPD, NPD, and NPDIT. Iteration-wise, PNPD and NPDIT exhibit similar behaviors, both converging faster than NPD. However, in terms of CPU time, PNPD shows a slight improvement over NPDIT.

\begin{figure}
    \centering
    \makebox[\textwidth][c]{
        \begin{subfigure}{\qPlotWidth}
            \centering
            \includegraphics[width=\textwidth]{\folder RRE_iterations.pdf}
        \end{subfigure}
        \begin{subfigure}{\qPlotWidth}
            \centering
            \includegraphics[width=\textwidth]{\folder RRE_time.pdf}
        \end{subfigure}
    }
    \makebox[\textwidth][c]{
        \begin{subfigure}{\qPlotWidth}
            \centering
            \includegraphics[width=\textwidth]{\folder SSIM_iterations.pdf}
        \end{subfigure}
        \begin{subfigure}{\qPlotWidth}
            \centering
            \includegraphics[width=\textwidth]{\folder SSIM_time.pdf}
        \end{subfigure}
    }
    \caption{Example \ex: Comparison of the RREs and SSIMs between PNPD, NPD,
    and NPDIT. The preconditioner parameter is $\nu=10^{-2}$.
    The number of nested loop iterations is $k_{\text{max}}=1$ for NPD
    and $k_{\text{max}}=2$ for NPDIT and PNPD.
    The regularization parameter is $\lambda=10^{-4}$ for NPD and
    NPDIT, and is $\lambda=6\cdot10^{-3}$ for PNPD.}
    \label{fig:example_\ex PNPD_comparison}
\end{figure}

Similarly to Figure \ref{fig:example_2PNPD_proposed} in Example 2, Figure \ref{fig:example_\ex PNPD_proposed} compares different parameter settings for the PNPD method.
Again, we observe that, when the optimal parameter choice is known, the stationary case remains the best among all possibilities.
In the stationary case, when choosing $\nu$ too small, PNPD becomes unable to achieve the same performance metrics as the optimal case. For example, PNPD with $\nu=10^{-2}$ achieves an SSIM of $0.935$ after 50 iterations, while PNPD and PNPD\_NE with $\nu=10^{-3}$ both achieve a lower SSIM of $0.92$.
Instead, the non-stationary PNPD with the bootstrap sequence~\eqref{eq:bootstrap_scheduler} (PNPD\_BT), achieves the same SSIM as the optimal stationary case.
PNPD\_BT, although slightly slower in the initial iterations compared to other considered cases, performs nearly as well as the optimal stationary case.

\def \folder{/example_\ex /PNPD_proposed/}
\begin{figure}
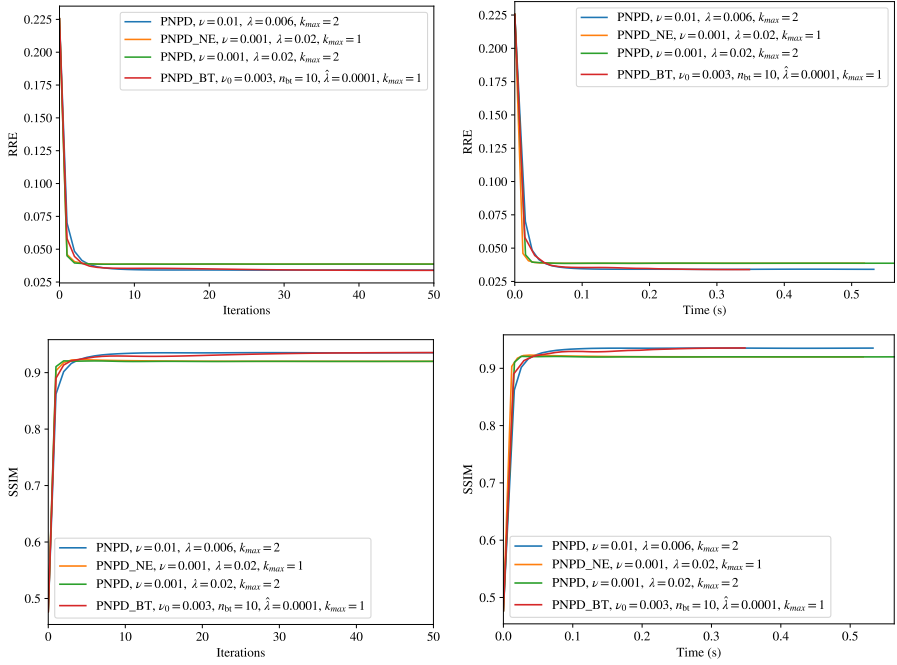

    \centering
    \makebox[\textwidth][c]{
        \begin{subfigure}{\qPlotWidth}
            \centering
            \includegraphics[width=\textwidth]{\folder RRE_iterations.pdf}
        \end{subfigure}
        \begin{subfigure}{\qPlotWidth}
            \centering
            \includegraphics[width=\textwidth]{\folder RRE_time.pdf}
        \end{subfigure}
    }
    \makebox[\textwidth][c]{
        \begin{subfigure}{\qPlotWidth}
            \centering
            \includegraphics[width=\textwidth]{\folder SSIM_iterations.pdf}
        \end{subfigure}
        \begin{subfigure}{\qPlotWidth}
            \centering
            \includegraphics[width=\textwidth]{\folder SSIM_time.pdf}
        \end{subfigure}
    }
    \caption{Example \ex: Comparison of the RREs and SSIMs between the proposed
    variants of PNPD. In particular, we show results obtained with PNPD,
    PNPD\_NE, and the bootstrap version of PNPD (PNPD\_BT)
    which uses the non-stationary sequence $\nu_n$ in
    \eqref{eq:bootstrap_scheduler} and $\lambda_n=\hat\lambda\cdot\|S_n^{{-1}}\|$.}
    \label{fig:example_\ex PNPD_proposed}
\end{figure}

\section{Conclusions}\label{sec:concl}
Inspired by the NPDIT method recently proposed in \cite{NPDIT}, we investigated preconditioning strategies for proximal gradient methods applied to image deblurring problems. We proved that, for this particular application, the NPDIT method can be interpreted as the right preconditioning. Therefore, we proposed a left preconditioning method to reduce the number of evaluations of the preconditioner, and thus the CPU time, at each iteration. Moreover, we explored some strategies to improve the stability of the proposed PNPD method preserving the fast convergence in the first iterations. Numerical results on image deblurring problems with white Gaussian noise confirm the advantages of PNPD over NPDIT. 

Interesting future investigations concern other choices of the preconditioning matrix, particularly when our proposal might be computationally expensive to invert, like in computed tomography applications. Moreover, the role of preconditioning combined with other extrapolation strategies, see \cite{buccini2020general}, deserves further investigation.

\section*{Acknowledgments}
The first, second, and fourth authors are partially supported by the PRIN project 2022ANC8HL funded by the Italian Ministry of University and Research.
The first and second authors are partially supported by the "INdAM - GNCS Project", code CUP\_E53C23001670001. This research was also supported by the Swiss National Science Foundation SNSF via the
projects Stress-Based Methods for Variational Inequalities in Solid Mechanics n. 186407 and ExaSolvers n. 162199.

\bibliographystyle{spmpsci}
\bibliography{bibliografia}

\end{document}